\documentclass[11pt]{article}

 \usepackage{amsmath,amssymb,amscd,amsthm,esint}
 
\usepackage{graphics,amsmath,amssymb,amsthm,mathrsfs}

\usepackage{graphics,amsmath,amssymb,amsthm,mathrsfs,amsfonts}

\usepackage{sidecap}
\usepackage{float}
\usepackage{extarrows}
\usepackage{booktabs}
\usepackage{verbatim}
\usepackage{hyperref}
\usepackage[usenames,dvipsnames]{xcolor}

\newtheorem{thm}{Theorem}[section]
\newtheorem{lemma}[thm]{Lemma}
\newtheorem{cor}[thm]{Corollary}

\theoremstyle{definition}
\newtheorem{remark}[thm]{Remark}

\def\XXint#1#2#3{{\setbox0=\hbox{$#1{#2#3}{\int}$}
         \vcenter{\hbox{$#2#3$}}\kern-.5\wd0}}

\def\R{\mathbb{R}}

\def\e{\varepsilon}

\numberwithin{equation}{section}

\begin{document}

 \bibliographystyle{amsplain}

\title{Large-Scale Lipschitz Estimates 
for Elliptic Systems\\  with Periodic  High-Contrast Coefficients}

\author{
Zhongwei Shen\thanks{Supported in part by NSF grant DMS-1856235.}}
\date{}
\maketitle

\begin{abstract}

This paper is concerned with the large-scale regularity in the homogenization of
 elliptic systems of elasticity with  periodic high-contrast coefficients.
 We obtain the large-scale Lipschitz estimate that is uniform with respect to the contrast ratio
 $\delta^2$ for $0<\delta<\infty$.
 Our study also covers the case of soft inclusions  ($\delta=0$)  as well as  the case of  stiff inclusions
($\delta=\infty$).
The large-scale Lipschitz estimate, together with classical local estimates,
allows us to establish explicit bounds for the matrix of fundamental solutions and its derivatives.

\medskip


\medskip

\noindent {\it MR (2020) Subject Classification}: 35B27;  74Q05.

\end{abstract}

%%%%%%%%%%%%%%%%%%%%%%%%%%%%%

\section{\bf Introduction}\label{section-1}

This paper is concerned with large-scale  regularity estimates in the homogenization of elliptic systems  of elasticity with 
periodic high-contrast coefficients.
Let $\omega$ be a connected and unbounded open set in $\mathbb{R}^d$. Assume that $\omega$ is 1-periodic; i.e.,
its characteristic function  is periodic with respect to $\mathbb{Z}^d$.
We also assume that 
each of connected components of
$
\mathbb{R}^d\setminus \omega 
$
is the closure of  a bounded open set $F_k$ with Lipschitz boundary,
and  that
\begin{equation}\label{dis}
\min_{k\neq \ell} \text{dist} (F_k, F_\ell ) >0.
\end{equation}
For $0< \delta < \infty$, define
\begin{equation}\label{Lambda-e}
\Lambda_{\delta} (x)
=\left\{
\aligned
& \delta  & \quad & \text{ if } x\in F=\cup_k F_k ,\\
& 1 & \quad & \text{ if } x\notin F.
\endaligned
\right.
\end{equation}
We are interested in the large-scale regularity estimates, that are uniform in $\delta>0$, for the elliptic operator 
\begin{equation}\label{op}
\mathcal{L}_\delta 
=-\text{\rm div} \big(\Lambda_{\delta^2} A\nabla \big).
\end{equation}
Here and thereafter  the coefficient matrix (tensor) $A=A(x)= (a_{ij}^{\alpha\beta} (x))$,
with $1\le \alpha, \beta, i, j \le d$, is assumed to be real, 
bounded measurable, 1-periodic,  and to satisfy the elasticity condition,
\begin{equation}\label{ellipticity}
\aligned
& a_{ij}^{\alpha\beta} (x) = a_{ji}^{\beta\alpha} (x) = a_{\alpha j}^{i\beta} (x)  ,\\
& \kappa_1 |\xi|^2 \le a_{ij}^{\alpha\beta} \xi_i^\alpha \xi_j^\beta \le \kappa_2 |\xi|^2
\endaligned
\end{equation}
for any symmetric matrix $\xi= (\xi_i^\alpha)\in \mathbb{R}^{d\times d}$,
where $\kappa_1$, $\kappa_2$ are positive constants.
Under these assumptions we will  show that if $u\in H^1(Q_R; \R^d)$ is a weak solution of
$\mathcal{L}_\delta (u)=0$ in a cube $Q_R= (-R/2, R/2)^d$ of size $R$ for some $R\ge 4$, then
\begin{equation}\label{Lip-e}
\sup_{1\le r\le R-3}
\left(\fint_{Q_r} |\nabla u|^2\right)^{1/2}
\le C \left(\fint_{Q_R\cap \omega} |\nabla u|^2 \right)^{1/2},
\end{equation}
with a constant $C$ independent of $R$ and $\delta$.
Let $ Du $ denote the symmetric gradient of $u$; i.e., 
$$
Du =  \big( \nabla u + (\nabla u)^T\big)/2,
$$
 where
$(\nabla u)^T$ denotes  the transpose of $\nabla u$. We also prove that for $R\ge 4$,
\begin{equation}\label{Lip-e-1}
\sup_{1\le r\le R-3}
\left(\fint_{Q_r } |D u|^2\right)^{1/2}
\le C \left(\fint_{Q_R\cap \omega } |D u|^2 \right)^{1/2}.
\end{equation}
We remark that the operator $\mathcal{L}_{\delta}$ arises naturally  in the modeling of acoustic propagations in porous media, diffusion processes  in highly heterogeneous media, and
 inclusions in composite materials \cite{Arbogast-1990, Allaire-1992, OSY-1992, JKO-1993}.

In the case $\delta=1$,
the regularity estimates for  the elliptic system $-\text{\rm div}(A(x/\e)\nabla)=f$
in the homogenization theory 
have been studied extensively in recent years (in this paper we have rescaled the equation so that the  microscopic scale $\e=1$ and the domain is large).
Using a compactness method,
the interior Lipschitz estimate and  the boundary Lipschitz estimate for the Dirichlet problem in
a $C^{1, \alpha}$ domain were established by M. Avellaneda and F. Lin  in a seminal work \cite{AL-1987}.
The boundary Lipschitz estimate for the Neumann problem in a $C^{1, \alpha}$ domain 
was obtained in \cite{KLS-2013-N}.
We refer the reader to  \cite{Shen-book} for further references on periodic homogenization, and to \cite{Armstrong-book}
for related work on the  large-scale regularity  in stochastic homogenization.

In this paper we will be concerned with the case $\delta \neq 1$,
where, in the simpler scalar case,
$\delta^2$ represents the conductivity ratio (or the ratio of diffusion coefficients) of the disconnected  {\it inclusions} $F=\cup_k F_k$ to the connected  \text{\it matrix}  $\omega$.
Notice that the operator $\mathcal{L}_{\delta}$  is elliptic, but neither  uniformly in $\delta \in (0, 1)$ nor in $\delta \in (1, \infty)$.
 We mention that in the scalar case  with  $0\le \delta<1$, $A=I$ and 
 $\omega$  being sufficiently smooth, using the compactness method in \cite{AL-1987}, 
the $W^{1, p}$ and Lipschitz estimates were  obtained  by L.-M. Yeh \cite{Yeh-2010, Yeh-2011, Yeh-2015, Yeh-2016}.
Also see earlier work in \cite{Schweizer-2000, Masmoudi-2004} for related uniform estimates in  the case $\delta=0$.
In \cite{Chase-Russell-2017}  B. Russell established the  large-scale interior Lipschitz estimate
for the system of elasticity  with bounded measurable coefficients in the case  $\delta=0$, using an approximation method
originated in \cite{Armstrong-Smart-2016}. 
The case $0<\delta<1$ was treated in \cite{Chase-Russell-2018}.
In the stochastic setting with $\delta=0$, S. Armstrong and P. Dario \cite{Armstrong-2018} obtained 
quantitative homogenization and large-scale regularity results for the random conductance model on a supercritical percolation.

The following is one of the main results of this paper.

\begin{thm}\label{main-thm-1}
Let $0\le \delta\le \infty$.
Assume that $A$ satisfies the elasticity condition \eqref{ellipticity} and is 1-periodic.
Let $u\in H^1(Q_R; \R^d)$ be a weak solution of $\mathcal{L}_\delta (u)=0$ in $Q_R$ for some $R\ge  4$.
Then \eqref{Lip-e}  and \eqref{Lip-e-1}
hold for some constant $C$ depending only on $d$, $\kappa_1$, $\kappa_2$, and $\omega$.
\end{thm}

Note that Theorem \ref{main-thm-1}  includes the  limiting  cases of periodically perforated domains: $\delta=0$ and $\delta=\infty$.
In the case $\delta=0$, which is referred to as the soft inclusions \cite{JKO-1993}, we call  $u\in H^1(\Omega; \R^d)$ 
is  a weak solution of $\mathcal{L}_0 (u)=f\chi_\omega $ in $\Omega$, if 
$$
\int_{\Omega\cap \omega} A\nabla u \cdot \nabla v  \, dx =\int_{\Omega\cap \omega}  f\cdot v \, dx
$$
for any $ v \in H_0^1(\Omega; \R^d)$.
Formally, this means 
that $ -\text{\rm div} (A\nabla u)=f$ in $\Omega \cap \omega$ and $\big(\frac{\partial u}{\partial \nu }\big)_-=0$ on 
$\Omega \cap \partial \omega$, where 
$
 \big( \frac{\partial u}{\partial \nu } \big)_- = n \cdot A (\nabla u )_-
$ denotes  the conormal derivative  taken from $\omega$ and 
$n$  the outward unit normal to $\partial F$.
For convenience we will also assume that $u$ is a weak solution of
$\text{\rm div} (A\nabla u)=0$ in $\Omega \cap F$.
In the case $\delta=\infty$, which is referred to as the stiff inclusions \cite{JKO-1993}, 
 a function $u$ in $H^1(\Omega; \R^d)$ is called a weak solution of
$\mathcal{L}_\infty (u)=f$ in $\Omega$ if
$ Du =0   \text{ in }\Omega \cap F$, and 
$$
\int_{\Omega\cap \omega} A\nabla u \cdot \nabla v \, dx =\int_\Omega f \cdot v\, dx
$$
for any $ v \in H_0^1(\Omega; \mathbb{R}^d)$ with
$D v =0$ in $\Omega  \cap F$.
This  implies that $-\text{\rm div}(A\nabla u)=f$ in $\Omega\cap \omega$ and that if $\overline{F}_k\subset \Omega$,
$$
\int_{\partial F_k}  \Big( \frac{\partial u}{\partial \nu}\Big)_-  \cdot \phi\, d\sigma =-\int_{F_k}  f \cdot \phi\, dx
$$
for any $\phi\in \mathcal{R}$, the space of rigid displacements.

The  large-scale uniform Lipschitz estimate in Theorem \ref{main-thm-1}, which holds under the assumptions that $A$ is bounded measurable and
$\partial \omega$ is locally Lipschitz, is new in the case $1<\delta\le \infty$,
even when $A$ is constant  and $\omega$ is smooth.
Under the additional conditions that  $\omega$ is locally $C^{1,\alpha}$ and $A$ is H\"older continuous,
\begin{equation}\label{smoothness}
|A(x)-A(y)|\le M_0 |x-y|^\sigma \quad \text{ for any } x, y\in \R^d,
\end{equation} 
where $M_0>  0$ and $\sigma \in (0, 1)$,
we may combine \eqref{Lip-e} with the local Lipschitz estimates for the operator $\mathcal{L}_\delta$
to obtain a true Lipschitz estimate.

\begin{thm}\label{main-thm-2}
Let $0\le \delta\le  \infty$ and $Q(x_0, R)=x_0 + Q_R$.
Assume that $A$ satisfies conditions \eqref{ellipticity}, \eqref{smoothness}, and is 1-periodic.
Also assume that each $F_k$ is a bounded $C^{1, \alpha}$ domain for some
$\alpha \in (0, 1)$.
Let $u\in H^1(Q(x_0, R); \R^d)$ be a weak solution of $\mathcal{L}_\delta (u)=0$ in 
$Q(x_0, R) $ for some $R\ge 4$.
Then 
\begin{equation}\label{f-Lip}
|\nabla u(x_0)|
\le C \left(\fint_{Q(x_0, R)\cap \omega } |\nabla u|^2 \right)^{1/2},
\end{equation}
where $C$ depends only on $d$, $\kappa_1$, $\kappa_2$, $\omega$, and $(\sigma, M_0)$ in \eqref{smoothness}.
\end{thm}

The  Lipschitz estimate \eqref{f-Lip} as well as its small-scale analogue  allows
us to construct a  $d\times d$ matrix $\Gamma_\delta  (x, y)$ of  fundamental solutions for the operator $\mathcal{L}_\delta$ in $\R^d$, and 
obtain its estimates that are uniform in $\delta \in (0, \infty)$.
In particular, we will show that  if  $d\ge 3$ and $1\le  \delta<\infty$,
\begin{equation}\label{fund-0}
\left\{
\aligned
 |\Gamma_\delta (x, y)| & \le C |x-y|^{2-d},\\
 |\nabla_x \Gamma_\delta (x,y)|
 +|\nabla_y \Gamma_\delta(x, y) |
 & \le C |x-y|^{1-d},\\
 |\nabla_x \nabla_y \Gamma_\delta (x, y)|  & \le C |x-y|^{-d}
 \endaligned
 \right.
 \end{equation}
 for any $x, y\in \R^d$ and  $x\neq y$,
 where $C$ depends only on $d$, $\kappa_1$, $\kappa_2$, $\omega$, and $(\sigma, M_0)$.
 In the case $0<\delta<1$, the estimates in \eqref{fund-0} continue to hold, provided that  either 
 $|x-y|_\infty \ge 4$ or $x, y\in \omega$.
 Here $|x-y|_\infty =\max ( |x_1-y_1|, \dots, |x_d -y_d|)$ denotes the $L^\infty$ norm in $\R^d$.
 See Theorems  \ref{thm-f-1}, \ref{thm-f-1a} and \ref{thm-f-3}.
 We mention that in the scalar case with $A=I $ and $0<\delta<1$, explicit bounds for fundamental solutions were obtained by L.-M. Yeh in \cite{Yeh-2016}.
 As in the case $\delta=1$ \cite{AL-1991, KS-2011-L, KLS-2014},
 estimates of fundamental solutions are  an important tool in  the study of optimal regularity problems in the homogenization theory  for solutions
 of $\mathcal{L}_\delta (u)=f$.
 In particular, it allows us to extend the  Lipschitz estimate \eqref{f-Lip}
 from solutions of $\mathcal{L}_\delta (u)=0$ to that of $\mathcal{L}_\delta (u)=f$.
 Indeed, under the same assumptions on $A$ and $\omega$ as in Theorem \ref{main-thm-2},
 we obtain 
 \begin{equation}\label{f-Lip-f}
|\nabla u(x_0)|
\le C_p \left\{  \left(\fint_{Q(x_0, R)\cap \omega } |\nabla u|^2 \right)^{1/2}
+R \left(\fint_{Q(x_0, R)} |f|^p \right)^{1/p} \right\}
\end{equation}
for $1\le \delta\le \infty$,
where $u$ is a weak solution of $\mathcal{L}_\delta (u)=f$ in $Q(x_0, R) $ for some
$R\ge 4$ and $p>d$.
If $0\le \delta<1$, the estimate (\ref{f-Lip-f}) holds for solutions of $\mathcal{L}_\delta (u)=f \chi_\omega$ in $Q(x_0, R)$.
See Theorem \ref{main-thm-3}.
 
 We now describe  our  general approach to the proof of Theorem \ref{main-thm-1}.
 As we mentioned earlier,
 the scalar case with $0\le \delta<1$ and $A=I$ was studied in \cite{Yeh-2010, Yeh-2011, Yeh-2015, Yeh-2016},
  using a compactness method of Avellaneda and Lin \cite{AL-1987}.
 The compactness argument is fairly complicated to implement  for the operator $\mathcal{L}_\delta$,
 as  both the coefficient matrix $A$ and the ratio $\delta^2$ should  be allowed to vary.
 A more direct approach, which originated in \cite{Armstrong-Smart-2016}, was used in \cite{Chase-Russell-2017, Chase-Russell-2018}
 to treat the case $0\le \delta<1$ with bounded measurable coefficients.
 The approach relies on a result on the convergence rate, uniform in $\delta$,  for the operator $-\text{\rm div}(\Lambda_{\delta^2} A(x/\e) \nabla )$
 as $\e\to 0$.
 It is not clear how to extend either of these two methods to the case $1<\delta\le  \infty$.
  In this paper we will adapt a more recent method of S. Armstrong, T. Kuusi, and C. Smart \cite{Armstrong-2020},
 which is  based on a Caccioppoli  type inequality and the fact that $\Delta_j u$ is a solution whenever $u$ is a solution,
   where $\Delta_j$ denotes  the   difference operator,
  \begin{equation}\label{diff}
  \Delta_j u (x) = u(x+e_j) -u(x)
  \end{equation}
  for $1\le j\le d$ and $e_j= (0, \dots, 1, \dots, 0)$ with $1$ in the $j^{\rm th}$ place.
  The basic idea is to transfer the higher-order regularity of $u$  in terms of the difference operator to 
  higher-order regularity of $u$ at a large scale through Caccioppoli and Poincar\'e's  inequalities.
  For elliptic systems the approach also uses a discrete Sobolev inequality.
  
  To carry out the approach described above, a key step is to establish a Caccioppoli inequality for solutions of $\mathcal{L}_\delta (u)=0$
  in $Q_R$ for $R$ large. In the case $0\le \delta\le  1$, it can be shown by an extension argument that 
\begin{equation}\label{C-0}
\int_{Q_{R/2}}  |\nabla u|^2\, dx \le \frac{C}{R^2} \int_{Q_R} |u|^2\, dx,
\end{equation}
which is more or less known  \cite{Chase-Russell-2017, Chase-Russell-2018}.
It is not known that  \eqref{C-0}  holds for the case $1< \delta\le \infty$, with constant $C$ independent of $\delta$.
However, if $\delta$ is sufficiently  large or $\delta=\infty$,
we are able to show that  for any $\ell\ge 1$ and $R\ge 32$,
\begin{equation}\label{C-1}
\int_{Q_{R/2 }}  |\nabla u|^2\, dx \le \frac{C_\ell }{R^2} \int_{Q_R} |u|^2\, dx + \frac{C_\ell}{R^{2\ell } }
\int_{Q_R} |\nabla u|^2\, dx,
\end{equation}
by some extension and iteration arguments.
It turns out that the weaker version  \eqref{C-1} with $\ell=1$, together with the discrete Sobolev inequality,  is sufficient to
complete the proof of \eqref{Lip-e}.
We point out that the method described above does not extend to the  nonhomogeneous  system $\mathcal{L}_\delta (u) =f$
with nonsmooth $f$. 
We resolve this issue by introducing the matrix of fundamental solutions.

The paper is organized as follows. In Section \ref{section-2} we give the proof of \eqref{C-0}.
The inequality \eqref{C-1} is proved in Section \ref{section-3}, while the proof of Theorem \ref{main-thm-1} is given in Section \ref{section-4}.
In Section  \ref{section-5} we collect some known results on local estimates and give the proof of Theorem \ref{main-thm-2}.
The matrix of  fundamental solutions is introduced and studied in Section \ref{section-6}.
Finally, we establish the Lipschitz estimate for solutions of $\mathcal{L}_\delta (u)=f$ in Section \ref{section-7}.

Recall that $Q_R =(-R/2, R/2)^d$ and $Q(x_0, R)=x_0 + Q_R$ for $R>0$ and $x_0\in \R^d$.
We use $\fint_E u=\frac{1}{|E|} \int_E u$ to denote the $L^1$ average of $u$ over a set $E$.
We use $C$ to denote a positive constant that may depend on $d$, $\kappa_1$, $\kappa_2$, and $\omega$.
If $C$  depends also on other parameters, it will be stated explicitly.
We emphasize that the results in Sections 2 - 4  hold with no smoothness condition on $A$ or $F=\R^d \setminus \overline{\omega}$ 
beyond that $A$ is bounded measurable and $F$ is locally Lipschitz.
In Sections 5 - 7 we impose  the H\"older continuity condition \eqref{smoothness} on $A$ and also assume that $F$ is locally $C^{1, \alpha}$.

%%%%%%%%%%%%%%%%%%%%%%%%%%%%%%%%%%%%%

\section{\bf Preliminaries}\label{section-2}

Throughout this paper we assume that $\omega$ is a connected, unbounded and 1-periodic open set in $\mathbb{R}^d$.
Write
\begin{equation}\label{F1}
\mathbb{R}^d\setminus \omega =\cup_k \overline{F}_k, 
\end{equation}
where each $\overline{F}_k$ is the closure of  a bounded Lipschitz domain $F_k$ with connected boundary.
We assume that $\{ \overline{F}_k \} $ are mutually disjoint and satisfy the condition (\ref{dis}).
This allows us to construct a sequence  of mutually disjoint open sets $\{\widetilde{F}_k \}$ with connected smooth  boundary such that $\overline{F}_k \subset \widetilde{F}_k$,
\begin{equation}\label{dist}
\left\{
\aligned
&  c_0 \le \text{dist}(\partial F_k, \partial \widetilde{F}_k), \\
& c_0 \le \text{dist} (\widetilde{F}_k, \widetilde{F}_\ell ) \text{ for } k\neq \ell,
\endaligned
\right.
\end{equation}
for some $c_0>0$.
Note that by the periodicity of $\omega$, 
$\{  F_k\}$  are the shifts of a finite number of bounded Lipschitz domains contained in $Q_2$.
As a result, 
we may  assume that $\{ \widetilde{F}_k \} $ are the shifts of a finite number of  bounded smooth domains contained in $Q_{5/2}$.

 Let $\mathcal{R}$ denote the space of rigid displacements of $\R^d$; i.e.,
 \begin{equation}\label{R}
 \mathcal{R}
 =\big\{ u  = E + Bx: \ E\in \R^d \text{ and } B^T=-B \big\},
 \end{equation}
 where $B^T$ denotes the transpose of the $d\times d$ matrix $B$.
  The following extension lemma will  be useful for us.

\begin{lemma}\label{lemma-ex}
Let $F_k$ and $\widetilde{F}_k$ be given above.
There exists a linear extension operator 
$$
P_k: H^1(\widetilde{F}_k \setminus \overline{F}_k; \R^d)
\to H^1(\widetilde{F}_k; \R^d)
$$ 
such that
\begin{align}
 & P_k (u)=u \quad \text{ for any } u \in \mathcal{R}, \label{ex-1}\\
 & \| P_k (u)\|_{H^1(\widetilde{F}_k)} 
 \le C \big( \| u \|_{L^2(\widetilde{F}_k \setminus \overline{F}_k)}
 + \|  Du \|_{L^2(\widetilde{F}_k \setminus \overline{F}_k)} \big), \label{ex-2}\\
 & \|\nabla P_k(u) \|_{L^2(\widetilde{F}_k)}
 \le C \| \nabla u\|_{L^2(\widetilde{F}_k \setminus \overline{F}_k)}, \label{ex-3}\\
 & \| D  P_k (u) ) \|_{L^2(\widetilde{F}_k)}
 \le C \|  D u  \|_{L^2(\widetilde{F}_k \setminus \overline{F}_k)}, \label{ex-4}
\end{align}
where $Du$ denotes the symmetric gradient of $u$ and $C$ depends only on $d$ and $\omega$.
\end{lemma}

\begin{proof} 
See \cite[pp.45-47]{OSY-1992}.
Note that since $\widetilde{F}_k$ and $F_k$ are shifts of a finite number of domains,
the constant  $C$ does not depend on $k$.
\end{proof}

Throughout the paper we  assume that $A$ is real, bounded measurable,  1-periodic,  and satisfies the elasticity condition \eqref{ellipticity}.
It is well known that  \eqref{ellipticity} implies 
\begin{align}
  A\xi \cdot \zeta  & \le \frac{\kappa_2}{4} |\xi +\xi^T| | \zeta +\zeta^T| \label{e-1},\\
 \frac{\kappa_1}{4} 
 |\xi +\xi^T|  & \le A \xi \cdot \xi \label{e-2}
 \end{align}
 for any $d\times d$ matrices $\xi$ and $\zeta$ \cite[pp.30-31]{OSY-1992}.
 
\begin{lemma}\label{lemma-2.1}
Let  $0<\delta <\infty$ and  $u\in H^1(\Omega; \R^d)$ be a weak solution of
$\mathcal{L}_\delta (u)=f$ in $\Omega$.
Then 
\begin{equation}\label{Ca-2.1}
\int_{\Omega} 
|\Lambda_\delta Du  |^2 | \varphi|^2 \, dx
\le C \int_{\Omega} |\Lambda_\delta  u|^2  |\nabla \varphi |^2\, dx
+ C \int_\Omega |f | |u| |\varphi|^2\, dx
\end{equation}
for any $\varphi \in C_0^1(\Omega)$,
where $C$ depends only $d$, $\kappa_1$ and $\kappa_2$.
In the case $\delta=0$, \eqref{Ca-2.1} holds for solutions of $\mathcal{L}_0 (u)=f\chi_\omega$ in $\Omega$.
\end{lemma}

\begin{proof}
Assume $0<\delta< \infty$.
Let $v =u\varphi^2$, where $\varphi \in C_0^1(\Omega)$.
Since
$$
\int_{\Omega} \Lambda_{\delta^2} A\nabla u \cdot \nabla v \, dx =\int_\Omega f \cdot v\, dx,
$$
we see that 
$$
\int_{\Omega}
\Lambda_{\delta^2} ( A\nabla u \cdot \nabla u) \varphi^2\, dx
=-2 \int_{\Omega} \Lambda_{\delta^2} (A\nabla u \cdot u (\nabla \varphi) ) \varphi\, dx
+\int_\Omega f \cdot v\, dx,
$$
from which the inequality \eqref{Ca-2.1} follows by using \eqref{e-1}-\eqref{e-2} and the Cauchy inequality.
The fact that $|A\nabla u|\le C |D u|$ is also needed.
The case $\delta=0$ may be handled in the same manner.
\end{proof}

\begin{lemma}\label{lemma-2.2}
Let $u\in H^1(\widetilde{F}_k; \R^d)$ be a weak solution of $-\text{\rm div}(A\nabla u)=f$ in $F_k$.
Then
\begin{align}
\int_{F_k} |\nabla u|^2\, dx &  \le C \int_{\widetilde{F}_k \setminus \overline{ F}_k} |\nabla u|^2\, dx
+ C \int_{{F}_k} |f|^2\, dx \label{2.2-0},\\
\int_{F_k} | Du  |^2\, dx &  \le C \int_{\widetilde{F}_k \setminus \overline{F}_k} | Du  |^2\, dx \label{2.2-1}
+C \int_{{F}_k} |f|^2\, dx,
\end{align}
where $C$ depends only on $d$, $\kappa_1$, $\kappa_2$, and $\omega$.
\end{lemma}

\begin{proof}
By Lemma \ref{lemma-ex}  there exists  $w\in H^1(\widetilde{F}_k; \R^d)$ such that $w=u$ on $\widetilde{F}_k \setminus F_k$ and
$$
 \| w\|_{H^1(\widetilde{F}_k) } \le C \| u\|_{H^1(\widetilde{F}_k \setminus \overline{F}_k)}.
 $$
Since $\text{\rm div} (A\nabla (u-w))=f-\text{\rm div}(A\nabla w)$ in $F_k$ and $u-w\in H^1_0(F_k; \R^d)$,
by the classical energy estimate,
$$
\aligned
\|\nabla u\|_{L^2(F_k)}
 & \le C \big\{ \| f\|_{L^2(F_k)}
+ \| \nabla w \|_{L^2(F_k)}  \big\}\\
& \le C \big\{ \| f\|_{L^2(F_k)} +
\| u\|_{H^1(\widetilde{F}_k\setminus \overline{F}_k)} \big\}.
\endaligned
$$
Note that for any $\phi\in \mathcal{R}$, 
$u-\phi$ satisfies the same condition as $u$.
It follows that 
\begin{equation}\label{2.2-2}
\|\nabla u -\nabla \phi \|_{L^2(F_k)}
\le C \big\{  \| f\|_{L^2(F_k)} +  \| u-\phi \|_{H^1(\widetilde{F}_k\setminus \overline{F}_k)} \big\}.
\end{equation}
By taking $\phi$ to be the $L^1$ average of $u$ over $\widetilde{F}_k \setminus \overline{F}_k$ and
using Poincar\'e's  inequality we obtain \eqref{2.2-0}.
To see \eqref{2.2-1}, we use
$$
\| D u \|_{L^2(F_k)}
\le \|\nabla u -\nabla \phi \|_{L^2(F_k)}
\le C\big\{ \| f\|_{L^2(F_k)} +  \| u-\phi \|_{H^1(\widetilde{F}_k\setminus \overline{F}_k)} \big\}.
$$
Since this holds for any $\phi\in \mathcal{R}$, 
 \eqref{2.2-1} follows by the second Korn inequality \cite[p.19]{OSY-1992}.
\end{proof}

\begin{remark}\label{remark-p}
It follows from Lemma \ref{lemma-2.2} that if $\mathcal{L}_\delta (u) =0 $ in $Q_{R+3}$
for some $R>0$, then
\begin{equation}\label{ex-1a}
\int_{Q_R} |\nabla u|^2\, dx
\le C \int_{Q_{R+3}\cap \omega} 
 |\nabla u|^2\, dx
 \quad
 \text{ and } \quad
 \int_{Q_R} | Du |^2\, dx
 \le C \int_{Q_{R+3}\cap \omega} | Du |^2\, dx.
 \end{equation}
 To see this, it suffices to note that if $F_k \cap Q_R\neq \emptyset$, then
 $\widetilde{F}_k \setminus \overline{F} _k \subset Q_{R+3}\cap \omega$.
 Also, observe that by Sobolev inequality,  for any $u\in H^1(\widetilde{F}_k; \R^d)$, 
 \begin{equation}\label{ex-2a}
 \int_{\widetilde{F}_k} |u|^2\, dx 
 \le C \int_{\widetilde{F}_k} |\nabla u|^2\, dx
 + C \int_{\widetilde{F}_k \setminus \overline{F}_k} |u|^2\, dx.
 \end{equation}
 This, together with \eqref{2.2-0}, implies that if 
 $\mathcal{L}_\delta (u)=0$ in $Q_{R+3}$ for some $R>0$, then
 \begin{equation}\label{ex-3a}
 \int_{Q_R} |u|^2\, dx
 \le C \int_{Q_{R+3}\cap \omega} 
 |u|^2\, dx
 + \int_{Q_{R+3}\cap \omega} 
 |\nabla u|^2\, dx.
 \end{equation}
\end{remark}

The next theorem gives a Caccioppoli inequality, which is uniform in $\delta \in [0, 1]$,
 for $\mathcal{L}_\delta$.
 
\begin{thm}\label{thm-2.3}
Suppose $0\le \delta < \infty $.
Let $u\in H^1(Q_{2R}; \R^d)$ be a weak solution of
$\mathcal{L}_\delta (u)=0$ in $Q_{2R} $ for some $R\ge 4$.
Then
\begin{equation}\label{Ca-1}
\int_{Q_{R}} |\nabla u|^2\, dx
\le \frac{C(1+\delta^2) }{R^2}
\int_{Q_{2R} } |u|^2\, dx,
\end{equation}
where $C$ depends only on $d$, $\kappa_1$, $\kappa_2$, and $\omega$.
\end{thm}

\begin{proof}

By the second Korn inequality,
\begin{equation}\label{K-1}
\int_{Q_R} 
|\nabla u|^2 \le C \int_{Q_R} | Du |^2\, dx
+\frac{C}{R^2} \int_{Q_R} |u|^2\, dx,
\end{equation}
where $C$ depends only on $d$.
In \eqref{Ca-2.1} we choose $\varphi\in C_0^1(Q_{2R})$ such that $\varphi=1$ in $Q_{R+3}$ and
$|\nabla \varphi|\le C/R$. This gives
\begin{equation}\label{2.3-1}
\int_{Q_{R+3} \cap \omega}
|D u  |^2\, dx \le \frac{C(1+\delta^2) }{R^2} \int_{Q_{2R}} |u|^2\, dx.
\end{equation}
which, together with \eqref{K-1} and  \eqref{ex-1a}, gives \eqref{Ca-1}.
\end{proof}

%%%%%%%%%%%%%%%%%%%%%%%%%%%%%%%%%%%%%%%%

\section{A Caccioppoli type inequality  for $1< \delta\le \infty$}\label{section-3}

We first consider the case $1<\delta<\infty$.

 \begin{lemma}\label{lemma-7.1}
 Let $ 1 < \delta< \infty$.
 Let $u\in H^1(\widetilde{F}_k; \R^d)$ be a weak solution of
 $-\text{\rm div}(\Lambda_{\delta^2} A\nabla u)=f$ in $\widetilde{F}_k$.
 Then
 \begin{equation}\label{7.1-0}
 \delta^2 \| Du\|_{L^2(F_k)}
 \le C \big\{ \| f\|_{L^p (\widetilde{F}_k)}
 + \| Du\|_{L^2(\widetilde{F}_k\setminus \overline{F}_k)} \big\},
 \ \end{equation}
 where  $p=\frac{2d}{d+2}$ for $d\ge 3$ and $p>1$ for $d=2$.
The constant  $C$ depends only on $d$, $\kappa_1$, $\kappa_2$, and $\omega$.
 \end{lemma}
 
 \begin{proof}
 Let $v\in H_0^1(\widetilde{F}_k; \R^d)$ be an extension of $u$ from $F_k$ to 
 $\widetilde{F}_k$ such that 
 \begin{equation}\label{7.1-1}
\|v\|_{H^1(\widetilde{F}_k)}
 \le C \| u\|_{H^1(F_k)}.
 \end{equation}
 Since
 $$
 \int_{\widetilde{F}_k \setminus \overline{F}_k}
 A\nabla u \cdot \nabla v\, dx
 +\delta^2 \int_{F_k} A\nabla u \cdot \nabla v\, dx=\int_{\widetilde{F}_k} f \cdot v\, dx,
 $$
 it follows that
 \begin{equation}\label{7.1-2}
 \aligned
 \delta^2\int_{F_k} |D u|^2\, dx
  & \le C \| f\|_{L^p(\widetilde{F}_k)}  \| v\|_{L^{p^\prime} (\widetilde{F}_k)}
 + C \|  Du \|_{L^2(\widetilde{F}_k \setminus \overline{F}_k)} \|  D v \|_{L^2(\widetilde{F}_k \setminus \overline{F}_k)}\\
 & \le C (   \| f\|_{L^p(\widetilde{F}_k)} 
 +  \|  Du \|_{L^2(\widetilde{F}_k \setminus \overline{F}_k)}) 
 \| u\|_{H^1(F_k)},
 \endaligned
 \end{equation}
 where we have used Sobolev inequality and \eqref{7.1-1}.
 We now choose $\phi\in \mathcal{R}$ such that
 $$
 \| u-\phi \|_{H^1(F_k)} \le C \| D u \|_{L^2(F_k)}.
 $$
 Since $u-\phi$ satisfies the same conditions as $u$, we may deduce \eqref{7.1-0},
  readily  from \eqref{7.1-2}, with $u-\phi$ in the place of $u$.
 \end{proof}

\begin{lemma}\label{lemma-3.2}
Suppose $1< \delta< \infty$.
Let $u\in H^1(Q_{R}; \R^d)$ be a weak solution of $\mathcal{L}_\delta (u)=0$ in $Q_{R}$
for some $R\ge 16$.
Then, for $(R/2)\le   r\le  R-8$ and $0<\e<1$,
\begin{equation}\label{3.2-0}
\int_{Q_r}
|\nabla u |^2\, dx
\le \frac{C}{\e (R-r)^2} \int_{Q_{R}} |u|^2\, dx
+   \Big(\e +\frac{C}{\delta^2} \Big) \int_{Q_{R}} |\nabla u|^2\, dx,
\end{equation}
where $C$ depends only on $d$, $\kappa_1$, $\kappa_2$, and $\omega$.
\end{lemma}

\begin{proof}
As in the proof of Theorem \ref{thm-2.3},
it follows from the second Korn inequality and \eqref{2.2-1} that
\begin{equation}\label{3.2-1}
\int_{Q_r} |\nabla u|^2\, dx
\le C \int_{Q_{r+3}\cap \omega} |Du |^2\, dx
+\frac{C}{r^2} \int_{Q_r} |u|^2\, dx.
\end{equation}
Since $r\ge R-r$, it suffices to bound the first term in the right-hand side of \eqref{3.2-1}.
To this end,
let $\varphi$ be a function in $C_0^1(Q_{R-3})$ such that
$\varphi=1$ in $Q_{r+3}$ and 
$$
|\nabla \varphi|\le C(R-r-6)^{-1}
\le C (R-r)^{-1},
$$
where we have used the assumption $R-r \ge   8$.
Recall  that if $F_k\cap Q_{R-3}\neq \emptyset$, then $\widetilde{F}_k\subset Q_{R}$.
For each $F_k$ with  $\widetilde{F}_k\subset Q_{R}$, we let $w_k\in H_0^1(\widetilde{F}_k; \R^d)$
be an extension of $u\varphi^2-g_k $ from $F_k$ to $\widetilde{F}_k$ with the property that
\begin{equation}\label{p-1}
\| w_k \|_{H^1(\widetilde{F}_k)} \le C \| u\varphi^2 -g_k \|_{H^1(F_k)},
\end{equation}
where $ g_k \in \mathcal{R} $ is to be determined.
Extend $w_k$ from $\widetilde{F}_k $ to $\R^d $ by zero and let
\begin{equation}\label{p-3}
\phi=u\varphi^2 -\sum_k  w_k \quad \text{ in } \R^d,
\end{equation}
where  the sum  is taken over those $k$'s for which $\widetilde{F}_k \subset Q_R$.
Note that $\phi (x) =g_k$ if $x\in F_k$ and $\widetilde{F}_k \subset Q_R$.
Since  $\phi\in H^1_0(Q_R; \R^d)$, we have 
\begin{equation}\label{p-2}
\int_{Q_R\cap \omega} A\nabla u \cdot \nabla \phi\, dx  +\delta^2 \int_{Q_R\cap F} A\nabla  u\cdot \nabla \phi\, dx=0.
\end{equation}
Since $D \phi =0$ in $F$, we obtain 
$$
\int_{Q_R\cap \omega} A\nabla u \cdot \nabla \phi\, dx=0.
$$
Thus, 
\begin{equation}\label{3.2-5}
\aligned
\Big|
\int_{Q_R\cap \omega} A\nabla u \cdot \nabla (u\varphi^2) \, dx\Big|
&\le \sum_k \Big|
\int_{\widetilde{F}_k \setminus \overline{F}_k}
A\nabla u \cdot \nabla w_k \, dx \Big|\\
& \le  C \sum_k
\|  D u\|_{L^2(\widetilde{F}_k \setminus \overline{F}_k) }
\| D w_k \|_{L^2(\widetilde{F}_k \setminus \overline{F}_k)}.
\endaligned
\end{equation}
Note that by \eqref{p-1},
$$
\|\nabla w_k\|_{L^2(\widetilde{F}_k \setminus \overline{F}_k)}
    \le C \| u\varphi^2 -g_k \|_{H^1(F_k)}
  \le C \| D(u \varphi^2)\|_{L^2(F_k)},
 $$
 where we have  chosen $g_k \in \mathcal{R}$ such that the last inequality holds.
 Consequently, 
 $$
 \aligned
 \|\nabla w_k\|_{L^2(\widetilde{F}_k \setminus \overline{F}_k)}
&\le C \| D u  \|_{L^2(F_k)}
+ C \|u \nabla \varphi\|_{L^2(F_k)}\\
& \le C \delta^{-2} \|  Du  \|_{L^2(\widetilde{F}_k \setminus \overline{F}_k)}
+ C \|u \nabla \varphi\|_{L^2(F_k)},
\endaligned
$$
where we have  used \eqref{7.1-0} for the last inequality.
This, together with \eqref{3.2-5}, gives
$$
\Big|
\int_{Q_R\cap \omega} A\nabla u \cdot \nabla (u\varphi^2) \, dx\Big|
\le (C \delta^{-2} +\e) \int_{Q_R} | D  u|^2\, dx
+ C \e^{-1}\int_{Q_R} |u |^2 | \nabla \varphi|^2\, dx
$$
for any $0< \e<1$, where we have used the Cauchy inequality.
Hence,
$$
\int_{Q_{r+3} \cap \omega} 
|Du |^2\, dx
 \le (C \delta^{-2} +\e) \int_{Q_R} | D u|^2\, dx
+ \frac{C}{\e (R-r)^2}
\int_{Q_R} |u|^2  \, dx,
$$
which, combined  with \eqref{3.2-1}, yields \eqref{3.2-0}.
\end{proof}

The following theorem provides a weaker version of the Caccioppoli inequality, that is uniform for $\delta\in (1, \infty)$,
for the operator $\mathcal{L}_\delta$.

\begin{thm}\label{thm-3.3}
Suppose $1<\delta< \infty$.
Let $u\in H^1(Q_R; \R^d)$ be a weak solution of
$\mathcal{L}_\delta (u)=0$ in $Q_R$ for some $R\ge 32$.
Then, for any $\ell\ge 1$,
\begin{equation}\label{3.3-0}
\int_{Q_{R/2}}  |\nabla u|^2\, dx
\le \frac{C}{R^2} \int_{Q_R} |u|^2\, dx
+\frac{C}{R^{2\ell} }  \int_{Q_R} |\nabla u|^2\, dx,
\end{equation}
where $C$ depends only on $d$, $\kappa_1$, $\kappa_2$, $\ell$,  and $\omega$.
\end{thm}

\begin{proof}
The proof uses Lemma \ref{lemma-3.2} and an iteration argument.
Let $r_i =R(1-2^{-i})$ for $i=1, 2, \dots$.
It follows from \eqref{3.2-0}  that for $0< \e<1$,
\begin{equation}\label{3.3-1}
\int_{Q_{r_i}}
|\nabla u|^2\, dx
\le \frac{C}{\e (r_{i+1} -r_i)^2}
\int_{Q_{r_{i+1}}}|u|^2\, dx
+ ( \e + C \delta^{-2}) \int_{Q_{r_{i+1}}} |\nabla u|^2\, dx,
\end{equation}
if $r_{i+1} \ge 16$ and 
$$
(1/2) r_{i+1} \le r_i \le  r_{i+1} -8.
$$
It is easy to verify that the conditions on $r_i$ are satisfied if $1\le i \le k$, where $k$ is the largest integer such that
$R2^{-k-1} \ge 8$.
Thus, by an induction argument,
$$
\int_{Q_{r_1}}
|\nabla u|^2\, dx
\le \frac{C_0}{\e}
\sum_{i=1}^k
\frac{ (\e + C_0 \delta^{-2})^{i-1}}{(r_{i+1} -r_i)^2}
\int_{Q_R} |u|^2\, dx
+ (\e + C_0 \delta^{-2})^k
\int_{Q_{r_{k+1}}} |\nabla u|^2\, dx,
$$
where $C_0$ depends only on $d$, $\kappa_1$, $\kappa_2$, and $\omega$.
Since $r_{i+1}-r_i= 2^{-i-1} R$, we see that 
$$
\aligned
\int_{Q_{R/2}} |\nabla u|^2\, dx
 & \le \frac{4C_0}{\e (\e + C_0 \delta^{-2}) R^2 }
\sum_{i=1}^k ( 4\e + 4 C_0 \delta^{-2})^i  \int_{Q_R} |u|^2\, dx\\
& \qquad\qquad
+ (\e + C_0 \delta^{-2})^k \int_{Q_R} |\nabla u|^2\, dx.
\endaligned
$$
We now choose $\e=  2^{-2\ell-2} $. It follows that if $4C_0\delta^{-2}\le 2^{-2\ell} $, then
$$
\int_{Q_{R/2}} |\nabla u|^2\, dx
\le \frac{C}{R^2} \int_{Q_R} |u|^2\, dx
+ (2^{-2\ell} )^k \int_{Q_R} |\nabla u|^2\, dx.
$$
This gives \eqref{3.3-0} for the case $\delta^2\ge  2^{2\ell +2} C_0$, as $2^{k}\approx R$.
Finally, we observe that the remaining case $1< \delta^2 <  2^{2\ell  +2} C_0$ is contained in Theorem \ref{thm-2.3}.
\end{proof}

We now consider the case $\delta=\infty$.
Recall that $u\in H^1(\Omega; \R^d)$ is called a weak solution of $\mathcal{L}_\infty (u)=0$ in $\Omega$ if
$Du =0$ in $\Omega\cap F$ and 
\begin{equation}\label{4.0-0}
\int_{\Omega\cap \omega} 
A\nabla u\cdot \nabla v \, dx =0
\end{equation}
for any $v \in H^1_0(\Omega; \R^d)$ with $D v =0$ in $\Omega\cap F$.

\begin{thm}\label{thm-4.1}
Let $u\in H^1(Q_R; \R^d)$ be a weak solution of $\mathcal{L}_\infty (u)=0$ in $Q_R$ for some $R\ge 32$.
Then, for any $\ell\ge 1$,
\begin{equation}\label{4.1-0}
\int_{Q_{R/2} } |\nabla u|^2\, dx
\le \frac{C}{R^2} \int_{Q_R} |u|^2\, dx
+ \frac{C}{R^{2\ell}} \int_{Q_R} |\nabla u|^2\, dx,
\end{equation}
where $C$ depends only on $d$, $\kappa_1$, $\kappa_2$, $\ell$, and $\omega$.
\end{thm}

\begin{proof}
In view of the proof of Theorem \ref{thm-3.3}, it suffices to show
that  for $(R/2)\le r \le R-8$ and $0<\e<1$,
\begin{equation}\label{4.1-1}
\int_{Q_r} |\nabla u|^2\, dx
\le \frac{C}{\e (R-r)^2} 
\int_{Q_R} |u|^2\, dx
+ \e \int_{Q_R} |\nabla u|^2\, dx.
\end{equation}
The proof of \eqref{4.1-1} is similar to that of Lemma \ref{lemma-3.2}.
Indeed, by the second Korn inequality,
\begin{equation}\label{4.1-2}
\int_{Q_r} |\nabla u |^2\, dx 
\le C \int_{Q_{r} \cap \omega} |D u |^2\, dx
+\frac{1}{r^2} \int_{Q_r} |u|^2\,dx,
\end{equation}
where we have used the fact $Du =0$ in $Q_r\cap F$.
Let $\varphi\in C_0^1(Q_{R-3})$ and $\phi\in H^1_0(Q_R; \R^d)$
be the same as in the proof of Lemma \ref{lemma-3.2}.
Note that $\phi|_{F_k}  \in \mathcal{R}$ for each $F_k$ (if $F_k \cap Q_{R-3}
=\emptyset $, then $\phi=0$).
This allows us to use \eqref{4.0-0} to obtain 
$$
\int_{Q_R\cap \omega}
A\nabla u\cdot \nabla \phi\, dx =0.
$$
The rest of the argument is the same as in the proof of Lemma \ref{lemma-3.2}, without the terms involving 
$C\delta^{-2}$. We omit the details.
\end{proof}

\begin{remark}\label{remark-3.1}
Let $1< \delta\le \infty$
and  $u\in H^1(Q_R; \R^d)$ be a weak solution of $\mathcal{L}_\delta (u)=0$ in $Q_R$
for some $R$ sufficiently large.
It follows from Theorems \ref{thm-3.3} and \ref{thm-4.1} (with $\ell=1$) that 
\begin{equation}\label{re-3.1}
\aligned
\sup_{s\le r\le R}
\left(\fint_{Q_r} |\nabla u|^2 \right)^{1/2}
 & \le 
 C \left(\fint_{Q_R} |\nabla u|^2\right)^{1/2} + C \sup_{s\le r\le R}
\inf_{E\in \R^d}
\frac{1}{r}
\left(\fint_{Q_r} |u-E|^2 \right)^{1/2}\\
& \qquad\qquad
+ \frac{C}{s}
\sup_{s\le r\le R}
\left(\fint_{Q_r} |\nabla u|^2 \right)^{1/2}
\endaligned
\end{equation}
for any $s\in [16, R]$,
where $C$ depends only on $d$, $\kappa_1$, $\kappa_2$, and $\omega$.
Choose $s$ so large that $Cs^{-1}\le (1/2)$.
This yields 
\begin{equation}\label{re-3.2}
\sup_{s\le r\le R}
\left(\fint_{Q_r} |\nabla u|^2 \right)^{1/2}
  \le 
 C \left(\fint_{Q_R} |\nabla u|^2\right)^{1/2} + C \sup_{s \le r\le R}
\inf_{E\in \R^d}
\frac{1}{r}
\left(\fint_{Q_r} |u-E|^2 \right)^{1/2}.
\end{equation}
Note that if $1\le r< s$, $|Q_r|^{-1/2} \| u\|_{L^2(Q_r)}
\le C |Q_s|^{-1/2} \| u\|_{L^2(Q_s)}$.
As a result, we obtain 
\begin{equation}\label{re-3.3}
\sup_{1 \le r\le R}
\left(\fint_{Q_r} |\nabla u|^2 \right)^{1/2}
  \le 
 C \left(\fint_{Q_R} |\nabla u|^2\right)^{1/2} + C \sup_{1\le r\le R}
\inf_{E\in \R^d}
\frac{1}{r}
\left(\fint_{Q_r} |u-E|^2 \right)^{1/2},
\end{equation}
where $C$ depends only on $d$, $\kappa_1$, $\kappa_2$, and $\omega$.
\end{remark}

%%%%%%%%%%%%%%%%%%%%%%%%%%%%%%

\section{Large-scale estimates}\label{section-4}

In this section we give the proof of Theorem \ref{main-thm-1}.
As we mentioned in Introduction, the approach is based on an idea from \cite{Armstrong-2020}.

Let  $u\in L^1(Q_{2r})$ for some $r\in \mathbb{N}$, define
\begin{equation}\label{5.0-0}
\widehat{u} (z) = \int_{Y+z} u(x)\, dx,
\end{equation}
where $Y=(0, 1)^d$, 
for any $z\in \mathbb{Z}^d $ such that $Y+z\subset Q_{2r}$.

\begin{lemma}\label{lemma-5.1}
Let $u\in H^1(Q_{2r})$ for some $r\in \mathbb{N}$.
Then
\begin{equation}\label{5.1-0}
\left(\fint_{Q_{2r} } |u|^2\right)^{1/2}
\le C \sup_{Y+z\subset Q_{2r}}  |\widehat{u}(z)|
+ C \left(\fint_{Q_{2r}} |\nabla u|^2\right)^{1/2},
\end{equation}
where $C$ depends only on $d$.
\end{lemma}

\begin{proof}
This follows by using  Poincar\'e's inequality on each unit cube $Y+z\subset Q_{2r}$ to obtain 
$$
\int_{Y+z} |u|^2\, dx \le |\widehat{u}(z)|^2 + C \int_{Y+z} |\nabla u|^2\, dx
$$
and summing the inequality over $z$.
%See e.g. \cite {Armstrong-2020}.
\end{proof}

For a  function $f$ defined in  $\mathbb{R}^d$ or $\mathbb{Z}^d$, let
\begin{equation}\label{5.2}
\Delta_j f (x)= f(x+e_j) -f(x)
\end{equation}
for $1\le j \le d$, where $e_j =(0, \dots, 1, \dots, 0)$ with $1$ in the $j^{th}$ position.
For a multi-index $\gamma=(\gamma_1, \gamma_2, \dots, \gamma_d)$,
we use the notation  $\Delta^\gamma f =\Delta_1^{\gamma_1} \Delta_2^{\gamma_2} \cdots \Delta_d^{\gamma_d} f$.
Let $\partial^k f = ( \Delta^\gamma f )_{|\gamma|= k}$ and 
$$
|  \partial^kf | =\Big(  \sum_{|\gamma|=k} |\Delta^\gamma f|^2 \Big)^{1/2}
$$
for an integer $k\ge 0$.
The following discrete Sobolev inequality will be needed:
\begin{equation}\label{Sob}
\sup_{z\in \mathbb{Z}^d\cap \overline {Q}_{2R}}
|f(z)|
\le C \sum_{k=0}^N 
R^k \left(\frac{1}{R^d}
\sum_{z\in \mathbb{Z}^d \cap \overline{Q}_{4R} }
|\partial^k f (z)|^2\right)^{1/2},
\end{equation}
where  $R\ge 1$ is an integer, $N=[d/2]+1$,  and $C$ depends only on $d$.
We refer the reader to \cite{Stevenson-1991} for a proof of (\ref{Sob}).

\begin{lemma}\label{lemma-5.3}
Let $u\in H^1(Q_{4R} )$ for some integer $R\ge 2$.
Then, for any integer $r\in [1, 2R]$,
\begin{equation}\label{5.3-0}
\inf_{E\in \mathbb{R}}
\left(\fint_{Q_{2r} } |u - E|^2\right)^{1/2}
\le {Cr}
\sum_{k=0}^N
R^k
\left(\fint_{Q_{4R} } 
|\nabla \partial ^k  u|^2\right)^{1/2}
+ C \left(\fint_{Q_{2r}} |\nabla u|^2\right)^{1/2},
\end{equation}
where $N=[d/2]+1$ and $C$ depends only on $d$.
\end{lemma}

\begin{proof}
We may assume $r\le R-1$; for otherwise,
\eqref{5.3-0}  (with $N=0$)  follows directly from Poincar\'e's inequality.
By \eqref{5.1-0} we have 
\begin{equation}\label{5.3-1}
\aligned
\left(\fint_{Q_{2r}} |u-\widehat{u} (0) |^2\right)^{1/2}
 & \le C \sup_{z\in \mathbb{Z}^d\cap \overline{Q}_{2r}}  |\widehat{u}(z)-\widehat{u} (0) |
+ C \left(\fint_{Q_{2r}} |\nabla u|^2\right)^{1/2}\\
 & \le C r \sup_{z\in \mathbb{Z}^d\cap \overline{Q}_{2r}}  |\partial  \widehat{u} (z)  |
+ C \left(\fint_{Q_{2r}} |\nabla u|^2\right)^{1/2}.
\endaligned
\end{equation}
To bound the first term in the right-hand side of (\ref{5.3-1}),
we use \eqref{Sob} to obtain 
\begin{equation}\label{5.3-2}
\sup_{z\in \mathbb{Z}^d\cap \overline{Q}_{2R-2} }  |\partial  \widehat{u} (z)  |
\le 
C \sum_{k=0}^N
R^k \left(\frac{1}{R^d}
\sum_{z\in \mathbb{Z}^d\cap \overline{Q}_{4R-4}}
|\partial^{k+1} \widehat{u}(z)|^2 \right)^{1/2}.
\end{equation}
Note that 
$$
\aligned
|\Delta_j  \partial^k \widehat{u} (z)|^2
 & \le \int_{Y+z}
|\partial^k u(x+e_j) -\partial^k u(x)|^2\, dx\\
&\le \int_{Y+z}
\int_0^1 |\nabla \partial^k u (x+t e_j)|^2 \, dt\, dx\\
&\le \int_{(0, 2)^d +z} |\nabla \partial^k u(x)|^2\, dx.
\endaligned
$$
It follows that
$$
\left(\frac{1}{R^d}
\sum_{z\in \mathbb{Z}^d\cap \overline{Q}_{4R-4}}
|\partial^{k+1} \widehat{u}(z)|^2 \right)^{1/2}
\le C \left(\fint_{Q_{4R}}
| \nabla  \partial ^k u|^2\, dx\right)^{1/2}.
$$
This, together with \eqref{5.3-1}and \eqref{5.3-2}, gives \eqref{5.3-0}.
\end{proof}

\begin{proof}[\bf Proof of Theorem \ref{main-thm-1}]

Let $u\in H^1(Q_R; \R^d)$ be a weak solution of $\mathcal{L}_\delta (u)=0$ in $Q_R$ for some $R\ge 4$.
Without loss of generality  we may assume $R$ is a large even integer.

We first point out that \eqref{Lip-e-1} follows from \eqref{Lip-e}.
Indeed, let $1\le r\le R-3$.  Note that for any $\phi \in \mathcal{R}$, we have $\mathcal{L}_\delta (u-\phi)=0$ in $Q_R$. 
It follows that
$$
\aligned
\left(\fint_{Q_r} |D u|^2 \right)^{1/2}
  &\le C \left(\fint_{Q_{R-3}} |\nabla (u-\phi)|^2\right)^{1/2}\\
  &\le C \left(\fint_{Q_{R-3}} |Du|^2 \right)^{1/2},
  \endaligned 
  $$
  where we have chosen $\phi\in \mathcal{R}$ so that the last inequality holds.
  This, together with \eqref{ex-1a}, gives \eqref{Lip-e-1}.
  The rest of the proof is devoted to (\ref{Lip-e}.
  In view of \eqref{ex-1a} it suffices to bound the left-hand side of \eqref{Lip-e})
  by the $L^2$ average of $|\nabla u|$ over $Q_R$.
  
\medskip

\noindent{\bf Case I:   $0 \le \delta\le 1$.}
Note that   $\mathcal{L}_\delta (\Delta^\gamma u )=0$ in $Q_{R-2k}$
for any multi-index $\gamma$ with $|\gamma|=k$.
It follows from Theorem \ref{thm-2.3} that
$$
\fint_{Q_{\rho }} |\nabla \partial^k  u |^2
\le \frac{C}{\rho ^2} \fint_{Q_{2\rho}}
| \partial^k   u|^2\, dx
$$
for $4\le \rho< (R-2k)/2$, where we have used the condition $0\le\delta\le 1$.
 This, together with the observation  that 
\begin{equation}\label{5.4-1}
\fint_{Q_{2\rho} } |\partial^k u|^2\, dx
\le C \fint_{Q_{2\rho+2}} |\nabla \partial^{k-1} u|^2\, dx
\end{equation}
for $k \ge 1$,   yields
\begin{equation}\label{5.4-2}
\fint_{Q_{\rho }} |\nabla \partial^k  u |^2
\le \frac{C}{\rho^2}  \fint_{Q_{2\rho+2}} |\nabla \partial^{k-1} u|^2\, dx.
\end{equation}
By induction we obtain
\begin{equation}\label{5.4-3}
\fint_{Q_{cR}}
|\nabla \partial^k u|^2\, dx
\le \frac{C}{\rho^{2k}}
\fint_{Q_R} |\nabla u|^2,
\end{equation}
where $C$ and $c$ depend only on $d$, $k$, $\kappa_1$, $\kappa_2$, and $\omega$.
By combining \eqref{5.4-3} with  \eqref{5.3-0} we see that
for any $r\in [1, cR]$,
\begin{equation}\label{5.4-4}
\inf_{E\in \R^d}
\frac{1}{r}
\left(\fint_{Q_r} |u-E|^2 \right)^{1/2}
\le C \left(\fint_{Q_R} |\nabla u|^2 \right)^{1/2}
+\frac{C}{r}
\left(\fint_{Q_r} |\nabla u|^2\right)^{1/2}.
\end{equation}
By Poincar\'e's inequality we see that the inequality above also holds for $r\in [cR, R]$.
Hence, if $1<s<R$, 
\begin{equation}\label{5.4-4-0}
\sup_{s\le r\le R} 
\inf_{E\in \R^d}
\frac{1}{r}
\left(\fint_{Q_r} |u-E|^2 \right)^{1/2}
\le C \left(\fint_{Q_R} |\nabla u|^2 \right)^{1/2}
+\frac{C}{s} \sup_{s\le r\le R} 
\left(\fint_{Q_r} |\nabla u|^2\right)^{1/2}.
\end{equation}
Note  that by Theorem \ref{thm-2.3}, 
\begin{equation}\label{5.4-4-00}
\sup_{s\le r\le R}
\left(\fint_{Q_r} |\nabla u|^2\right)^{1/2}
\le C \left(\fint_{Q_R} |\nabla u|^2\right)^{1/2}
+ C \sup_{s\le r \le R}
\inf_{E\in \R^d}
\frac{1}{r} 
\left(\fint_{Q_r} |u-E|^2\right)^{1/2}.
\end{equation}
Thus,
$$
\sup_{s\le r\le R}
\left(\fint_{Q_r} |\nabla u|^2\right)^{1/2}
\le C \left(\fint_{Q_R} |\nabla u|^2 \right)^{1/2}
+\frac{C}{s} \sup_{s\le r\le R} 
\left(\fint_{Q_r} |\nabla u|^2\right)^{1/2}.
$$
Choose $s>1$ so large that $C s^{-1}\le (1/2)$.
This leads to 
$$
\sup_{s\le r\le R}
\left(\fint_{Q_r} |\nabla u|^2\right)^{1/2}
\le C \left(\fint_{Q_R} |\nabla u|^2 \right)^{1/2}.
$$
The estimate for the case $1\le  r< s$ follows from the case $r=s$.

\medskip

\noindent{\bf Case II: $1<\delta\le \infty$.}
As in the case $0\le \delta\le 1$, $\mathcal{L}_\delta (\Delta^\gamma u)=0$ in $Q_{R-2k}$ for any multi-index $\gamma$
with $|\gamma|=k$.
It follows from Theorems \ref{thm-3.3} and \ref{thm-4.1} (with $\ell=1$) that
\begin{equation}\label{5.4-10}
\fint_{Q_{\rho}}
|\nabla  \partial^k u|^2
\le \frac{C}{\rho^2}
\fint_{Q_{2\rho} } | \partial^k u|^2 +\frac{C}{\rho^2}
\fint_{Q_{2\rho} } |\nabla \partial^k u  |^2
\end{equation}
if $16 \le \rho\le  (R-2k)/2$.
This, together with \eqref{5.4-1} and  a simple observation that
$$
\fint_{Q_{2\rho}}
|\nabla \partial^k u|^2
\le C \fint_{Q_{2\rho+2}}
|\nabla \partial^{k-1} u|^2,
$$
gives (\ref{5.4-2}).
As a result, the inequality (\ref{5.4-4-0})
continues to hold for the case $1<\delta\le \infty$.
In view of Remark \ref{remark-3.1},
the inequality (\ref{5.4-4-00}) also holds for $1<\delta\le \infty$. 
The rest of the proof is the same as in Case I.
\end{proof}

It follows from Theorem \ref{main-thm-1} and Poincar\'e's inequality  that if $\mathcal{L}_\delta (u)=0$ in $Q_R$ for some 
$R\ge 1$, then
\begin{equation}\label{L-inf}
\sup_{1\le r\le R} \left(\fint_{Q_r} |u|^2\right)^{1/2}
\le C \left(\fint_{Q_R} |u|^2 \right)^{1/2}
+ C R^2 \left(\fint_{Q_R} |\nabla  u|^2\right)^{1/2},
\end{equation}
where $C$ depends only on $d$, $\kappa_1$, $\kappa_2$, and $\omega$.

%%%%%%%%%%%%%%%%%%%%%%%%%%%%%%%%%%%%%%%

\section{Local Lipschitz  estimates}\label{section-5}

Throughout this section we will assume $A$ satisfies the elasticity condition (\ref{ellipticity}) and
H\"older continuity condition (\ref{smoothness}).
The periodicity condition is not needed.
For $0<r\le 4$, let
\begin{equation}\label{B}
\aligned
Q_r^+  & =\big\{ x=(x^\prime, x_d) \in Q_r: x_d> \psi (x^\prime) \big\},\\
Q_r^-  & =\big\{ x=(x^\prime, x_d) \in Q_r: x_d< \psi (x^\prime) \big\},\\
I_r  &  =\big\{ x=(x^\prime, x_d) \in Q_r: x_d= \psi (x^\prime) \big\},\\
\endaligned
\end{equation}
where $\psi: \R^{d-1} \to \R$ is a $C^{1, \sigma}$ function for some $\sigma \in (0, 1)$
such that $\psi(0)=0$ and $\|\nabla \psi \|_{C^{1, \sigma}(\R^{d-1})} \le M_1$.
Let $0<\delta<\infty$ and  $u\in H^1(Q_r; \R^d)$ be a solution of 
\begin{equation}\label{local-1}
\left\{
\aligned
 -\text{\rm div} (A\nabla u) & = \delta^{-2} f  &\quad & \text{ in } Q_r^+,\\
 -\text{\rm div} (A\nabla u)  & =f  & \quad &  \text{ in } Q_r^-,\\
 \delta^2  \frac{\partial u}{\partial \nu_+}  & = \frac{\partial u}{\partial \nu_-}  &\quad & \text{ on }  I_r,
\endaligned
\right.
\end{equation}
where $\frac{\partial u}{\partial \nu_\pm}
=n\cdot A(\nabla u)_\pm$ and $\pm$ indicates the limit taken from $Q_r^\pm$, respectively.
If $\delta=0$, by a solution of \eqref{local-1}, we mean that $-\text{\rm div} (A\nabla u)=f$ in $Q_r^-$,
$-\text{\rm div} (A\nabla u)=0$ in $Q_r^+$, and that $\frac{\partial u}{\partial \nu_-} =0$ on $I_r$.
If $\delta=\infty$, the equation \eqref{local-1} is understood in the sense that $u|_{Q_r^+} \in \mathcal{R}$ and 
$-\text{\rm div}(A\nabla u)=f$ in $Q_r^-$.

\begin{lemma}\label{local-lemma-1}
Assume that $A$ satisfies \eqref{ellipticity} and \eqref{smoothness}.
Let $0\le \delta\le 1$ and $u\in H^1(Q_r; \R^d)$ be a weak solution of \eqref{local-1}  for some $0<r \le 4$.
Then
\begin{equation}\label{local-2}
\left\{
\aligned
\|\nabla u\|_{L^\infty (Q_{r/2}^-) }
&\le \frac{C}{r}
\left\{ 
 \delta^2 \left(\fint_{Q_r^+} |u|^2 \right)^{1/2}
+  \left(\fint_{Q_r^-} |u|^2 \right)^{1/2}
+r^2  \left( \fint_{Q_r} |f|^p \right)^{1/p}\right\},\\
\| \nabla u\|_{L^\infty (Q_{r/2}^+ ) }
 & \le \frac{C}{r}
\left\{
\left(\fint_{Q_r} |u|^2\right)^{1/2}
+ r^2 \left(\fint_{Q^-_r} |f|^p \right)^{1/p} 
+ \delta^{-2} r^2 
\left(\fint_{Q_r^+} |f|^p \right)^{1/p} \right\},
\endaligned
\right.
\end{equation}
where $p>d$ and $C$ depends only on $d$, $ p$, $\kappa_1$, $\kappa_2$, $\sigma$, $M_0$, and $M_1$.
\end{lemma}

\begin{proof}
The case $0< \delta_0\le \delta\le 1$ with a constant $C$ depending on $\delta_0$ 
follows from the classical  results on Lipschitz estimates for elliptic systems 
with piecewise H\"older continuous  coefficients.
Indeed, since $\psi$ is $C^{1, \sigma}$,  the problem may be reduced to the case
$\psi=0$ by flatting the boundary.
One may further reduce the problem to the case of constant coefficients by a Campanato perturbation argument.
We refer the reader to \cite{Li-2000, Dong-2012}  and their  references for more recent development.

We now  treat the case where $\delta$ is small.
By rescaling we may assume $r=1$.
Let $0<\rho<1$ and $0<\tau < \sigma$. 
Since $\text{\rm div} (A\nabla u)=f$ in $Q_1^-$,
by the classical $C^{1, \tau} $ estimates for Neumann problems,
\begin{equation}\label{local-3}
\| \nabla u\|_{C^{\tau}(Q^-_{\rho/2})}
\le C\Big \|\frac{\partial u}{\partial \nu_-}\Big \|_{C^{\tau}  (I_{\rho})}
+\frac{C}{\rho^{1+\tau} }
\left\{
\left(\fint_{Q_\rho^-} |u|^2\right)^{1/2}
+ \rho^2 \left(\fint_{Q_\rho^-} |f|^p \right)^{1/p} \right\}.
\end{equation}
Let $1/2\le s<t<1$.
By covering  $Q_{s}^-$ with cubes of size $\rho/2=c(t-s)$ and applying (\ref{local-3}) and interior $C^{1, \tau}$ estimates,
we may deduce that
\begin{equation}\label{local-4}
\aligned
\| \nabla u\|_{C^{\tau} (Q_s^-)}
 & \le C \Big\| \frac{\partial u}{\partial \nu_-} \Big\|_{C^\tau (I_{t_1} )}
+ C (t-s)^{-\frac{d}{2} -1-\tau} \left\{
 \left(\fint_{Q_1^-} |u|^2\right)^{1/2} + \left(\fint_{Q_1^-} |f|^p \right)^{1/p} \right\}\\
&\le C \delta^2 \| \nabla u\|_{C^\tau(Q_{t_1} ^+)}
+ C (t-s)^{-\frac{d}{2} -1-\tau} \left\{
 \left(\fint_{Q_1^-} |u|^2\right)^{1/2} + \left(\fint_{Q_1^-} |f|^p \right)^{1/p} \right\},\\
\endaligned
\end{equation}
where  $t_1=(t+s)/2$
and
we have used the relation $\frac{\partial u}{\partial\nu_-} =\delta^2 \frac{\partial u}{\partial \nu_+}$ on $I_1$ for the last inequality.
Since $-\text{\rm div}(A\nabla u)=\delta^{-2} f $ in $Q_1^+$ for $\delta>0$,
a similar argument using $C^{1, \tau}$ estimates for the Dirichlet problem gives
\begin{equation}\label{local-5}
 \| \nabla u\|_{C^\tau (Q_{t_1}^+)}
\le C  \|\nabla u\|_{C^\tau (Q_t^-)}
+ C (t-s)^{-\frac{d}{2}-1-\tau }
\left\{
  \left(\fint_{Q_1^+} |u|^2\right)^{1/2}
 +\delta^{-2}  \left( \fint_{Q_1^+} |f|^p \right)^{1/p} \right\}.
\end{equation}
By combining (\ref{local-4}) with (\ref{local-5}) it follows that 
\begin{equation}\label{local-6}
\aligned
\| \nabla u\|_{C^{\tau} (Q_s^-)}
 & \le  C \delta^2  \|\nabla u\|_{C^\tau (Q_t^-)}
+ C (t-s)^{-\frac{d}{2}-1-\tau }
\left\{
 \left(\fint_{Q_1^-} |u|^2\right)^{1/2}
 + \delta^2
\left(\fint_{Q_1^+} |u|^2\right)^{1/2} \right\}.\\
&\qquad\qquad
+ C (t-s)^{-\frac{d}{2} -1 -\tau} 
\left(\fint_{Q_1} |f|^p \right)^{1/p},
\endaligned
\end{equation}
which also holds for the case $\delta=0$.
 Let $s_i=1-2^{-i}$.
 By taking $s=s_i$ and $t=s_{i+1}$ in \eqref{local-6} and using iteration,
 we see that
 \begin{equation}\label{local-7}
 \aligned
 &  \| \nabla u\|_{C^\tau (Q_{1/2}^-)}\\
 & \le 
 C \sum_{i=1}^N (s_{i+1}-s_i)^{-\frac{d}{2} -1-\tau} \delta^{2 (i-1)}
 \left\{
 \left(\fint_{Q_1^-} |u|^2\right)^{1/2}
 + \delta^2
\left(\fint_{Q_1^+} |u|^2\right)^{1/2}
+ \left(\fint_{Q_1} |f|^p \right)^{1/p}  \right\}\\
& \qquad\qquad
+ C \delta^{2N} \| \nabla u\|_{C^{\tau} (Q_{s_{N+1}})}
\endaligned
\end{equation}
for any $N\ge 1$.
Note that $s_{i+1}-s_i= 2^{-i-1}$.
It follows that if $2^{\frac{d}{2} +1 +\tau} \delta^2 \le (1/2)$, we may let $N \to \infty$ in \eqref{local-7}
to obtain the first inequality in \eqref{local-2} with $r=1$.
The second inequality follows from the the  first and \eqref{local-5}.
\end{proof}

\begin{remark}\label{remark-5.1a}
We may replace (\ref{local-2}) by 
\begin{equation}\label{local-2b}
\left\{
\aligned
\|\nabla u\|_{L^\infty (Q_{r/2}^-) }
&\le  C 
\left\{ 
 \delta^2 \left(\fint_{Q_r^+} |\nabla u|^2 \right)^{1/2}
+  \left(\fint_{Q_r^-} |\nabla u|^2 \right)^{1/2}
+r  \left( \fint_{Q_r} |f|^p \right)^{1/p}\right\},\\
\| \nabla u\|_{L^\infty (Q_{r/2}^+ ) }
 & \le  C 
\left\{
\left(\fint_{Q_r} |\nabla u|^2\right)^{1/2}
+ r  \left(\fint_{Q^-_r} |f|^p \right)^{1/p} 
+ \delta^{-2} r
\left(\fint_{Q_r^+} |f|^p \right)^{1/p} \right\},
\endaligned
\right.
\end{equation}
To see this, in the proof of Lemma \ref{local-lemma-1},
one  replaces $u$ in \eqref{local-3} and (\ref{local-4}) by $u-E$ and applies 
Poincar\'e's inequality.
\end{remark}

\begin{remark}\label{remark-5.1}
Under the same assumptions as in Lemma \ref{local-lemma-1}, we have 
\begin{equation}\label{local-2-r}
\left\{
\aligned
\| u\|_{L^\infty (Q_{r/2}^-) }
&\le C
\left\{ 
 \delta^2 \left(\fint_{Q_r^+} |u|^2 \right)^{1/2}
+  \left(\fint_{Q_r^-} |u|^2 \right)^{1/2}
+r^2  \left( \fint_{Q_r} |f|^p \right)^{1/p}\right\},\\
\|  u\|_{L^\infty (Q_{r/2}^+ ) }
 & \le  C 
\left\{
\left(\fint_{Q_r} |u|^2\right)^{1/2}
+ r^2 \left(\fint_{Q^-_r} |f|^p \right)^{1/p} 
+ \delta^{-2} r^2 
\left(\fint_{Q_r^+} |f|^p \right)^{1/p} \right\},
\endaligned
\right.
\end{equation}
This follows readily from \eqref{local-2} and the Mean Value Theorem.
\end{remark}

The next lemma treats the case $1<\delta\le \infty$.

\begin{lemma}\label{local-lemma-2}
Assume that $A$ satisfies (\ref{ellipticity} and (\ref{smoothness}).
Let $1<\delta\le \infty$ and $u\in H^1(Q_r; \R^d)$ be a weak solution of \eqref{local-1}  for some $0<r \le 4$.
Then
\begin{equation}\label{local-2-0}
\left\{
\aligned
\|\nabla u\|_{L^\infty (Q_{r/2}^+) }
&\le \frac{C}{r}
\left\{ 
 \delta^{-2} \left(\fint_{Q_r^-} |u|^2 \right)^{1/2}
+  \left(\fint_{Q_r^+} |u|^2 \right)^{1/2}
+  \delta^{-2} r^2 \left(\fint_{Q_r} |f|^p \right)^{1/p} 
\right\},\\
\| \nabla u\|_{L^\infty (Q_{r/2}^- ) }
 & \le \frac{C}{r}
 \left\{
\left(\fint_{Q_r} |u|^2\right)^{1/2}
+ \delta^{-2} r^2 \left(\fint_{Q_r^+} |f|^p \right)^{1/p}
+ r^2 \left(\fint_{Q_r^-} |f|^ p \right)^{1/p} \right\},
\endaligned
\right.
\end{equation}
where $C$ depends only on $d$,  $p$, $\kappa_1$, $\kappa_2$, $\sigma$, $M_0$, and $M_1$.
\end{lemma}

\begin{proof}
The case $1<\delta< \infty$ follows from  the proof of Lemma \ref{local-lemma-1}
by interchanging $Q_r^+$ with $Q_r^-$ and $\delta^2$ with $\delta^{-2}$.
Recall that if $\delta=\infty$, $u|_{Q_r^+} \in \mathcal{R}$.
 As a result, the first inequality in \eqref{local-2-0} holds by a simple
rescaling, while the second  follows from the Lipschitz estimate for the Dirichlet problem.
\end{proof}

\begin{remark}\label{remark-5.2b}
One may replace \eqref{local-2-0} by 
\begin{equation}\label{local-2-0b}
\left\{
\aligned
\|\nabla u\|_{L^\infty (Q_{r/2}^+) }
&\le C 
\left\{ 
 \delta^{-2} \left(\fint_{Q_r^-} |\nabla u|^2 \right)^{1/2}
+  \left(\fint_{Q_r^+} |\nabla u|^2 \right)^{1/2}
+  \delta^{-2} r \left(\fint_{Q_r} |f|^p \right)^{1/p} 
\right\},\\
\| \nabla u\|_{L^\infty (Q_{r/2}^- ) }
 & \le C 
 \left\{
\left(\fint_{Q_r} |\nabla u|^2\right)^{1/2}
+ \delta^{-2} r  \left(\fint_{Q_r^+} |f|^p \right)^{1/p}
+ r \left(\fint_{Q_r^-} |f|^ p \right)^{1/p} \right\},
\endaligned
\right.
\end{equation}
This follows from Remark \ref{remark-5.1}.
\end{remark}

\begin{remark}\label{remark-5.2}
It follows from \eqref{local-2-0} that if $1<\delta\le \infty$,
\begin{equation}\label{local-3-0}
\left\{
\aligned
\|\nabla u\|_{L^\infty (Q_{r/2}^+) }
&\le C
\left\{ 
 \delta^{-2} \left(\fint_{Q_r^-} |u|^2 \right)^{1/2}
+  \left(\fint_{Q_r^+} |u|^2 \right)^{1/2}
+  \delta^{-2} r^2 \left(\fint_{Q_r} |f|^p \right)^{1/p} 
\right\},\\
\| \nabla u\|_{L^\infty (Q_{r/2}^- ) }
 & \le C
 \left\{
\left(\fint_{Q_r} |u|^2\right)^{1/2}
+ \delta^{-2} r^2 \left(\fint_{Q_r^+} |f|^p \right)^{1/p}
+ r^2 \left(\fint_{Q_r^-} |f|^ p \right)^{1/p} \right\}.
\endaligned
\right.
\end{equation}
\end{remark}

The following theorem provides the local Lipschitz estimate for
solutions of $\mathcal{L}_\delta (u)=f$ in $Q(x_0, r)=x_0 +Q_r$
(if $\delta=0$, we assume $f=0$ in $Q(x_0, r)\cap F$).

\begin{thm}\label{local-thm-0}
Assume that $A$ satisfies \eqref{ellipticity} and \eqref{smoothness}.
Let $0\le \delta\le \infty$ and $u\in H^1(Q (x_0, r); \R^d)$ be a weak solution of  $\mathcal{L}_\delta (u)=f$ in 
$Q(x_0, r) $  for some $x_0\in \R^d$ and  $0<r\le 2$,
where $f\in L^p (Q(x_0, r); \R^d)$ for some $p>d$.
Then 
\begin{equation}\label{5.5-0}
\aligned
 & |u(x_0)| + r|\nabla u (x_0)|\le \\
 &
 \left\{
 \aligned
  &  C \left(\fint_{Q(x_0, r) } |\Lambda_{\delta^2} u|^2 \right)^{1/2} 
 + Cr^2  \left(\fint_{Q(x_0, r)} | f|^p \right)^{1/p} &  & \text{ if } 0\le \delta \le 1 \text{ and } x_0 \in \omega,\\
&   C  \left(\fint_{Q(x_0, r) } |u|^2 \right)^{1/2} 
+ Cr^2 \left(\fint_{ Q(x_0, r)}
|\Lambda_{\delta^{-2}} f|^p \right)^{1/P} & & \text{ if }  0\le \delta \le 1 \text{ and } x_0 \in F,
\endaligned
\right.
\endaligned
\end{equation}
and
\begin{equation}\label{5.5-1}
\aligned
& |u(x_0)| + r|\nabla u (x_0)|\le \\
 & 
 \left\{
 \aligned
  & C \left(\fint_{Q(x_0, r) } | u|^2 \right)^{1/2}
 + C r^2
 \left(\fint_{Q(x_0, r)} | \Lambda_{\delta^{-2}}  f|^p \right)^{1/p} & & \text{ if } 1< \delta < \infty \text{ and } x_0 \in \omega,\\
&  \frac{ C}{\delta^2 }  \left(\fint_{Q(x_0, r) } |\Lambda_{\delta^2} u|^2 \right)^{1/2} 
+ C\delta^{-2} r^2
\left(\fint_{Q_r} |f|^p \right)^{1/p}  &  & \text{ if }  1< \delta < \infty \text{ and } x_0 \in F.
\endaligned
\right.
\endaligned
\end{equation}
If $\delta=\infty$, we have
\begin{equation}\label{5.5-2}
\aligned
|u(x_0)| + r|\nabla u (x_0)|
 & \le C \left(\fint_{Q(x_0, r) } | u|^2 \right)^{1/2} 
 + C r^2 \left(\fint_{Q(x_0, r)\cap \omega} |f|^p \right)^{1/p}& \quad & \text{ if } x_0 \in \omega
 \\
|u(x_0)| + r|\nabla u (x_0)|
& \le  C   \left(\fint_{Q(x_0, r)\cap F  } | u|^2 \right)^{1/2}  & \quad & \text{ if }   x_0 \in F.
\endaligned
\end{equation}
The constant $C$ depends only on $d$, $p$, $\kappa_1$, $\kappa_2$, $\omega$, and $(\sigma, M_0)$ in \eqref{smoothness}.
\end{thm}

\begin{proof}
Note that $-\text{\rm div}(A\nabla u)=f$ in $Q(x_0, r)\cap \omega$ and 
$-\text{\rm div}(A\nabla u)=\delta^{-2}f$ in $Q(x_0, r)\cap F$.
If $Q(x_0, cr)\subset \omega$ or $F$ for some small $c>0$, the estimates in \eqref{5.5-0}-\eqref{5.5-2} follow directly from the interior estimates 
for solutions of $-\text{\rm div}(A\nabla u)=f $. 
In the case $Q(x_0, cr)\cap \partial \omega \neq \emptyset$,
one may find $y_0\in \partial \omega$ such that $x_0\in Q(y_0, r/2)$ and $Q(y_0, r/2)\subset Q(x_0, r)$.
As a result,  the estimates in \eqref{5.5-0}-\eqref{5.5-2}  follow readily  from \eqref{local-2}, \eqref{local-2-r},
\eqref{local-2-0} and \eqref{local-3-0} by a simple localization argument.
\end{proof}

\begin{proof}[\bf Proof of Theorem \ref{main-thm-2}]

Note that for all cases in Theorem \ref{local-thm-0} with $f=0$,
\begin{equation}\label{5.6-0}
|u(x_0)| + r|\nabla u(x_0)| \le C \left(\fint_{Q(x_0, r)} |u|^2\right)^{12}.
\end{equation}
Since $u-E$ is also a solution for any $E\in \R^d$, one may use Poincar\'e's inequality to obtain 
\begin{equation}\label{5.6-1}
|\nabla u(x_0)| \le C \left(\fint_{Q(x_0, r)} |\nabla u|^2\right)^{12}
\end{equation}
for $0<r\le 1$.
Thus, if $u$ is a weak solution of $\mathcal{L}_\delta (u)=0$ in $Q (x_0, R)$ for some $R\ge 4$,
then
$$
|\nabla u(x_0)|
  \le C \left( 
\fint_{Q(x_0, 1) } |\nabla u|^2 \right)^{1/2}\\
 \le C 
 \left( 
\fint_{Q(x_0, R)\cap \omega } |\nabla u|^2 \right)^{1/2},
$$
where we have used Theorem \ref{main-thm-1} for the last inequality.
\end{proof}

%%%%%%%%%%%%%%%%%%%%%%%%%%%%%%%%

\section{Estimates of fundamental solutions}\label{section-6}

Throughout this section we assume that  $d\ge 3$, $0<\delta<\infty$, and that $A$ satisfies \eqref{ellipticity}, \eqref{smoothness}, and is 1-periodic.
We also assume that each $F_k$ is a $C^{1, \sigma} $ domain for some $\sigma \in (0, 1)$.
Under these conditions, by combining the large-scale estimates in Section \ref{section-4}  with the  local  Lipschitz estimates in Section \ref{section-5}, 
we see that if  $u\in H^1(Q(x_0, R); \R^d)$ is a weak solution of
$\mathcal{L}_\delta (u)=0$ in $Q (x_0; R)$ for some $R>0$, then 
$$
|\nabla u(x_0)|\le C_\delta \left(\fint_{Q(x_0, R)} |\nabla u|^2 \right)^{1/2},
$$
where $C_\delta$ may depend on $\delta$. It follows that the operator $\mathcal{L}_\delta$ possesses 
a fundamental solution $\Gamma_\delta (x, y) = \big( \Gamma_\delta^{\alpha \beta} (x, y)  \big)_{d\times d}$ 
in the sense that if 
\begin{equation}\label{f-1}
u(x)=\int_{\R^d} \Gamma_\delta (x, y)  f (y)\, dy
\end{equation}
for some $f\in C_0^\infty (\R^d; \R^d)$,  then  $u\in L^{2^*} (\R^d; \R^d)$, $\nabla u\in L^2(\R^d; \R^{d\times d})$,  and 
\begin{equation}\label{f-2}
\int_{\R^d} \Lambda_{\delta^2} A\nabla u \cdot \nabla  v \, dx
=\int_{\R^d} f \cdot v \, dx
\end{equation}
for any $v  \in L^{2^*} (\R^d; \R^d)$ with $\nabla v \in L^2 (\R^d; \R^d)$, where $2^*=\frac{2d}{d-2}$.
Moreover, there exists a constant $C_\delta$ such that
$|\Gamma_\delta (x, y)|\le C_\delta |x-y|^{2-d}$ and
$|\nabla_x \Gamma_\delta (x, y)| + |\nabla_y \Gamma_\delta  (x, y)|\le C_\delta  |x-y|^{1-d}$
for any $x, y\in \R^d$.
We refer the reader to \cite{Hofmann-2007} for the construction of $\Gamma_\delta(x, y)$
under a H\"older continuity condition on weak solutions.
Our goal of  this section is to establish  the explicit dependence of $C_\delta$ on $\delta$.

\begin{lemma}\label{lemma-f-1}
Let $u$ be given by \eqref{f-1} with $f\in C_0^\infty(\omega; \R^d)$.
Then
\begin{equation}\label{f-1-0}
\delta \| Du  \|_{L^2( F)}  
+  \|\nabla u \|_{L^2(\omega)} +\| u\|_{L^{p^\prime} (\omega)} 
 \le C \| f\|_{L^p(\omega)},
\end{equation}
where $p=\frac{2d}{d+2}$ and $C$ depends only on $d$, $\kappa_1$, $\kappa_2$, and $\omega$.
 \end{lemma}
 
 \begin{proof}
  By letting $v=u$ in \eqref{f-2}  we obtain 
 \begin{equation}\label{f-1-1}
 \delta^2 \int_{F} |D u |^2\, dx  
 +\int_\omega |Du |^2\, dx \le C \| f\|_{L^p(\omega)} \| u\|_{L^{p^\prime } (\omega)},
 \end{equation}
 where $p=\frac{2d}{d+2}$. 
 Next, let $U$   be an extension of $u$ from $\omega$ to $\R^d$ such that
 \begin{equation}\label{f-1-3}
 \| \nabla U\|_{L^2(\R^d)} \le C \| \nabla u\|_{L^2(\omega)} \quad \text{ and } \quad
  \|  D U \|_{L^2(\R^d)} \le C \|  D u  \|_{L^2(\omega)} .
  \end{equation}
 The function $U$ may be obtained by extending $u$ from $\widetilde{F}_k \setminus F_k$ to $F_k$, for each $k$,
 so that 
 \begin{equation}\label{f-1-3a}
 \| \nabla U\|_{L^2(\widetilde{F}_k) } \le C \| \nabla u\|_{L^2(\widetilde{F} _k \setminus \overline{F}_k)}
 \quad \text{ and } \quad
 \|  D U \|_{L^2(\widetilde{F}_k) } \le C \|  D u \|_{L^2(\widetilde{F} _k \setminus \overline{F}_k)}.
 \end{equation}
 See Lemma \ref{lemma-ex}.
  Since $|u(x)| +|x| |\nabla u (x)| =O(|x|^{2-d})$ as $|x|\to \infty$ and $d\ge 3$, we see that
 \begin{equation}\label{decay}
 \frac{1}{R^2} \int_{Q_{2R} \setminus Q_R}  \big ( |u|^2 +|\nabla u|^2 \big) \, dx \to 0 \quad \text{ as } R \to \infty.
 \end{equation}
  The property (\ref{ex-2}) also implies  that $U$ satisfies the condition (\ref{decay}).
  This allows us to  apply the first  Korn inequality and Sobolev inequality in $Q_R$  and then let $R\to \infty$  to deduce that
 \begin{equation}\label{f-1-1a}
 \| \nabla U\|_{L^2(\R^d)} \le C \| D U \|_{L^2(\R^d)}
 \quad
 \text{ and }
 \quad
 \| U\|_{L^{p^\prime}(\R^d)} \le C \| \nabla U  \|_{L^2(\R^d)}.
\end{equation}
 As a result, we obtain 
 \begin{equation}\label{f-1-4}
 \| U \|_{L^{p^\prime}(\R^d)}
 \le C \| \nabla U \|_{L^2(\R^d)}
 \le C \| D U \|_{L^2(\R^d)}
 \le C \| D u  \|_{L^2(\omega)}.
 \end{equation}
 It follows that
 \begin{equation}\label{f-1-4a}
 \aligned
  & \| \nabla u\|_{L^2(\omega)} =\| \nabla U \|_{L^2(\omega)} \le C \| D u \|_{L^2(\omega)},\\
  & \| u\|_{L^{p^\prime}(\omega)}= \| U  \|_{L^{p^\prime}(\omega)}
  \le C \| D u  \|_{L^2(\omega)}.
  \endaligned
 \end{equation}
 Consequently, by \eqref{f-1-1} and the Cauchy inequality,  
 we see that $ \delta \| D u \|_{L^2(F)} \le C \| f\|_{L^p(\omega)}$ and
 $$
 \|\nabla u\|_{L^2(\omega)}
 + \| u\|_{L^{p^\prime} (\omega)} 
 \le C \| Du \|_{L^2(\omega)} 
 \le C \| f\|_{L^p(\omega)},  
  $$
  which completes the proof.
 \end{proof}
 
 \begin{remark}\label{re-div}
 Let $u$ be given by (\ref{f-1}) with $f=\text{\rm div} (g)$,
 where $g\in C_0^\infty (\omega; \R^{d\times d})$.
 Then
 $$
 \delta^2 \int_{F} | Du |^2\, dx 
 + \int_{\omega} |Du |^2\, dx \le C \| g \|_{L^2(\omega)} \| \nabla u\|_{L^2(\omega)}.
 $$
 Using (\ref{f-1-4a}), we obtain 
 \begin{equation}\label{re-div-0}
 \delta
 \| Du \|_{L^2(F)} + \| \nabla u\|_{L^2(\omega)}
 + \| u\|_{L^{p^\prime} (\omega)}
 \le C \| g\|_{L^2(\omega)},
 \end{equation}
 where $p^\prime=\frac{2d}{d-2}$ and $C$ depends only on $d$, $\kappa_1$, $\kappa_2$, and $\omega$.
 \end{remark}
 
 \begin{remark}\label{re-large}
 
 Suppose $1\le \delta< \infty$.
 Let $u$ be given by \eqref{f-1} with $f\in C_0^\infty (\R^d; \R^d)$.
 By letting $v=u$ in \eqref{f-2} we obtain 
 $\| D u \|^2_{L^2(\R^d)} \le C \| f\|_{L^p(\R^d)} \| u\|_{L^{p^\prime}(\R^d)}$.
 Using 
 $$\| u\|_{L^{p^\prime}(\R^d)} \le C \| \nabla u\|_{L^2(\R^d)} \le C \| D u \|_{L^2 (\R^d)},
 $$
 we see that  $\|u\|_{L^{p^\prime}(\R^d)} \le C \| f\|_{L^p(\R^d)}$.
 
 \end{remark}
 
 \begin{thm}\label{thm-f-1}
 Let $0<\delta<\infty$.
 For $x, y\in \R^d$ with $|x-y|_\infty\ge 4$, we have
 \begin{align}
 |\Gamma_\delta (x, y)|  & \le C |x-y|^{2-d}  , \label{f-2-0}\\
 |\nabla_x \Gamma_\delta (x, y)|
 +|\nabla _y \Gamma_\delta  (x, y)|
 & \le  C |x-y|^{1-d} , \label{f-2-1} \\
 |\nabla_x\nabla_y \Gamma_\delta (x, y)| & \le C |x-y|^{-d}, \label{f-2-00}
 \end{align}
 where $C$ depends only on $d$, $\kappa_1$, $\kappa_2$, $\omega$, and $(\sigma, M_0)$.
 \end{thm}
 
\begin{proof}
Fix $x_0, y_0\in \R^d$ with $r=|x_0-y_0|_\infty\ge 4$.
Let $u$ be given by \eqref{f-1} with $f\in C_0^\infty (\omega\cap Q(y_0, r); \R^d) $.
Since $\mathcal{L}_\delta (u) =0$ in $Q (x_0, r)$,
it follows from (\ref{5.6-0}) and \eqref{L-inf} that
\begin{equation}\label{f-2-2}
\aligned
|u(x_0)|
& \le C \left(\fint_{Q(x_0, 1/2)} |u|^2\right)^{1/2}\\
& \le C \left(\fint_{Q(x_0, r/4)}
|u|^2 \right)^{1/2}
+ C r \left(\fint_{Q(x_0, r/4)} |\nabla u|^2 \right)^{1/2}\\
&  \le C \left(\fint_{Q(x_0, 3+r/4)\cap \omega}
|u|^2 \right)^{1/2}
+ C r \left(\fint_{Q(x_0, 3+ r/4) \cap \omega} |\nabla u|^2 \right)^{1/2},
\endaligned
\end{equation}
where we have used \eqref{ex-1a}  and \eqref{ex-3a} for the last inequality.
We now use (\ref{f-1-0}) to bound the right-hand side of \eqref{f-2-2}.
This gives
$$
\aligned
| u(x_0)|
&\le C r^{1-\frac{d}{2}} \big\{ \| u\|_{L^{p^\prime}(\omega)}
+ \| \nabla u\|_{L^2(\omega)} \big\}\\
&\le C r^{1-\frac{d}{2}} \| f\|_{L^p(\omega)},
\endaligned
$$
where $p=\frac{2d}{d+2}$.
By duality it follows that 
\begin{equation}\label{f-2-3}
\left(\int_{\omega\cap Q(y_0, r)}
| \Gamma_\delta (x_0, y)|^{p^\prime} \, dy \right)^{1/p^\prime}
\le C r^{1-\frac{d}{2}}.
\end{equation}
Note that if $f=\text{\rm div} (g)$, where $g\in C_0^\infty(\omega\cap Q(y_0, r); \R^{d\times d} )$,
we may use  \eqref{f-2-2} and \eqref{re-div-0} to obtain 
$$
|u(x_0)|\le C r^{1-\frac{d}{2}} \| g\|_{L^2(\omega)}.
$$
By duality we deduce that
\begin{equation}\label{f-2-4}
\left(\int_{\omega\cap Q(y_0, r)}
| \nabla_y \Gamma_\delta (x_0, y)|^{2} \, dy \right)^{1/2}
\le C r^{1-\frac{d}{2}}.
\end{equation}
Also, note that by Theorem \ref{main-thm-2},
$$
|\nabla u(x_0)|\le C \left(\fint_{\omega\cap Q(x_0, r)} |\nabla u|^2 \right)^{1/2}
\le Cr^{-\frac{d}{2}} \| g \|_{L^2(\omega)}.
$$
Again, by duality, we obtain 
\begin{equation}\label{dd}
\left(\int_{\omega\cap Q(y_0, r)}
| \nabla_x \nabla_y \Gamma_\delta (x_0, y)|^{2} \, dy \right)^{1/2}
\le C r^{-\frac{d}{2}}.
\end{equation}

Now, let $ v(y)= \Gamma_\delta (x_0, y)$.
Then  $\mathcal{L}_\delta^* (v)=0$ in $Q(y_0, r)$, where $\mathcal{L}_\delta^*$ denotes the adjoint of
$\mathcal{L}_\delta$.
Since $\mathcal{L}_\delta^*$ satisfies the same conditions as $\mathcal{L}_\delta$,
we may use \eqref{f-2-2} to obtain
$$
\aligned
|v(y_0)|
& 
\le C \left(\fint_{\omega\cap Q(y_0, 3+r/4)}
|v|^2 \right)^{1/2}
+ C r \left(\fint_{\omega\cap Q(y_0, 3+ r/4) } |\nabla v|^2 \right)^{1/2}\\
& \le C r^{2-d},
\endaligned
$$
which gives \eqref{f-2-0}. Also, note that by Theorem \ref{main-thm-2},
\begin{equation}\label{f-2-5}
|\nabla v (y_0)|
   \le C \left(\fint_{\omega\cap Q(y_0, r) } |\nabla v|^2 \right)^{1/2}.
\end{equation}
This, together with \eqref{f-2-4}, gives 
$|\nabla_y \Gamma_\delta (x_0, y_0)|\le C r^{1-d}$.
The  estimate $|\nabla_x \Gamma_\delta (x_0, y_0)|\le C r^{1-d}$ 
follows from the fact that the fundamental solution $\Gamma^*_\delta (x, y)$  for $\mathcal{L}_\delta^*$ is given by the
transpose of $\Gamma_\delta (y, x)$.
Finally, the estimate for $\nabla_x \nabla_y \Gamma_\delta (x_0, y_0)$
follows from \eqref{dd} and the fact that
$\mathcal{L}_\delta^* (\nabla_x \Gamma_\delta (x_0, \cdot))=0$ in $\R^d\setminus \{ x_0 \}$.
\end{proof}

Next, we treat the case where $ 1\le \delta< \infty$ and $|x-y|_\infty < 4$.

 \begin{thm}\label{thm-f-1a}
 Suppose $ 1\le \delta<\infty$. Then estimates \eqref{f-2-0}, \eqref{f-2-1} and \eqref{f-2-00}
 continue to hold
 for $x, y\in \R^d$ with $|x-y|_\infty< 4$.
   \end{thm}

\begin{proof}
The proof is similar to that of Theorem \ref{thm-f-1}.
Fix $x_0, y_0\in\R^d$ with $r=|x_0-y_0|_\infty \le 4$.
Let $u$ be given by \eqref{f-1} with $f\in C_0^\infty(Q(y_0, r); \R^d)$.
Since $\mathcal{L}_\delta (u)=0$ in $Q(x_0, r)$,
in view of  \eqref{5.5-1} and Remark \ref{re-large}, we obtain 
\begin{equation}\label{f-2-2a1}
\aligned
|u(x_0)| + r |\nabla u (x_0)|
&\le C \left( \fint_{Q(x_0, r)}
|u|^2\right)^{1/2}\\
& \le C r^{1-\frac{d}{2}} \| f\|_{L^p(\R^d)}.
\endaligned
\end{equation}
By duality it follows that
\begin{equation}\label{f-2-2a2}
\left(\int_{Q(y_0, r)}
|\Gamma_\delta (x_0, y)|^{p^\prime}\, dy \right)^{1/p^\prime}
\le C r^{1-\frac{d}{2}}.
\end{equation}
Since $\mathcal{L}_\delta^* (\Gamma_\delta (x_0, \cdot))
=0$ in $Q(y_0, r)$,  the desired estimates
follow readily from the first inequality in (\ref{f-2-2a1}). 
We omit the details.
\end{proof}

It remains to handle the case  where $0<\delta<1$ and $|x-y|_\infty<4$.

 \begin{lemma}\label{lemma-f-2}
 Let  $0< \delta\le 1$ and 
 \begin{equation}\label{f-20-0}
 u(x)=\int_{\R^d} \Gamma_\delta (x, y) \Lambda_\delta (y) f(y)\, dy
 \end{equation}
 for some $f\in C_0^\infty(\R^d; \R^d)$.
 Then
 \begin{equation}\label{f-20-1}
 \|\Lambda_\delta u\|_{L^{p^\prime} (\R^d)}
 + \|  \Lambda_\delta \nabla u \|_{L^2(\R^d)}
 \le C \| f\|_{L^{p} (\R^d)},
 \end{equation}
 where $p =\frac{2d}{d+2}$ and $C$ depends only on $d$, $\kappa_1$, $\kappa_2$, and $\omega$.
 \end{lemma}
 
\begin{proof}
As in the proof of Lemma \ref{lemma-f-1}, we have
$$
\| u\|_{L^{p^\prime} (\R^d)} \le C \| \nabla u\|_{L^2(\R^d)} 
\le C \| D u \|_{L^2(\R^d)}
$$
 and 
$\| u\|_{L^{p^\prime} (\omega)} +\| \nabla u\|_{L^2(\omega)} \le C \|  Du \|_{L^2(\omega)}.
$
It  follows that 
\begin{equation}\label{f-2-3x}
\aligned
\| \Lambda_\delta u\|_{L^{p^\prime} (\R^d)} 
+ \|\Lambda_\delta \nabla u\|_{L^2(\R^d)} 
 & \le C\delta  \|  Du \|_{L^2(\R^d)} 
+ C \| Du  \|_{L^2(\omega)}\\
& \le C \|\Lambda_\delta D u  \|_{L^2(\R^d)},
\endaligned
\end{equation}
where we have used  the assumption $\delta\le 1$ for the last inequality.
By letting $v=u$ in \eqref{f-2}, we obtain 
$$
\|\Lambda_\delta Du \|^2_{L^2(\R^d)} \le C \| \Lambda_\delta u \|_{L^{p^\prime} (\R^d)} \| f\|_{L^p(\R^d)},
$$
which, together with \eqref{f-2-3x}, yields  \eqref{f-20-1}.
\end{proof}

\begin{thm}\label{thm-f-3}
Suppose $0<\delta< 1$.
For $x, y\in \R^d$ with $|x-y|_\infty<  4 $, we have
 \begin{align}
 \Lambda_\delta (x) \Lambda_\delta (y) |\Gamma_\delta (x, y)|  & \le
 C |x-y|^{2-d}   , \label{f-3-0}\\
 \Lambda_\delta (x) \Lambda_\delta (y) 
 \big\{ |\nabla_x \Gamma_\delta (x, y)|
 +|\nabla _y \Gamma_\delta  (x, y)| \big\} 
 & \le  C |x-y|^{1-d} , \label{f-3-1}\\
 \Lambda_\delta (x) \Lambda_\delta (y) |\nabla_x \nabla_y \Gamma_\delta (x, y)| 
 & \le C |x-y|^{-d}, \label{f-3-1x} 
 \end{align}
 where $C$ depends only on $d$, $\kappa_1$, $\kappa_2$,   $\omega$, and $(\sigma, M_0)$.
 \end{thm}
 
 \begin{proof}
 
 Fix $x_0, y_0\in \R^d$ with $r=|x_0-y_0|_\infty<4$.
 Let $u$ be given by (\ref{f-20-0}) with $f\in C_0^\infty(Q(y_0, r); \R^d) $.
 Then $\mathcal{L}_\delta (u)=0$ in $Q(x_0, r)$.
 It follows from (\ref{5.5-0}) that
 \begin{equation}\label{f-3-2}
 \aligned
 \Lambda_\delta (x_0)
 \big\{ |u(x_0)| + r|\nabla u(x_0) | \big\}
 &\le C \left(\fint_{Q(x_0, r)}
 |\Lambda_\delta u|^2 \right)^{1/2}\\
 &\le Cr^{1-\frac{d}{2}}
 \| f\|_{L^p(\R^d)},
 \endaligned
 \end{equation}
 where we have used \eqref{f-20-1} for the last inequality.
 By duality this implies that
 \begin{equation}\label{f-3-3}
 \Lambda_\delta  (x_0)
 \left(\int_{Q(y_0, r)}
 |\Lambda_\delta (y) \Gamma_\delta (x_0, y)|^{p^\prime} \, dy \right)^{1/p^\prime}
 \le C r^{1-\frac{d}{2}}.
 \end{equation}
 Since $\mathcal{L}_\delta^* ( \Gamma_\delta (x_0, \cdot)) =0$ in $\R^d\setminus \{ x_0\}$, we may use the first inequality in \eqref{f-3-2}
 and (\ref{f-3-3}) 
 to obtain
 \begin{equation}\label{f-3-4}
 \Lambda_\delta (x_0)\Lambda_\delta (y_0)
 |\Gamma_\delta (x_0, y_0)|,
 \le C r^{2-d}.
 \end{equation}
 which gives \eqref{f-3-0}.
 The estimates in \eqref{f-3-1} also follow from the first inequality in \eqref{f-3-2}
 and (\ref{f-3-3}). 
 Finally, \eqref{f-3-1x} follows from the first  inequality in \eqref{f-3-2} and \eqref{f-3-1}.
 \end{proof}
  
We end   this section  with  a decay estimate  of  $D\Gamma_\delta (x, y)$ for $x\in F$, as $\delta  \to \infty$.
  
 \begin{thm}\label{thm7.2}
 Let $1\le \delta<\infty$.
 Let $u\in H^1(Q(x_0, R); \R^d)$ be a weak solution of
 $\mathcal{L}_\delta (u) =0$ in $Q(x_0, R)$
 for some $x_0\in F$ and $R\ge 5$.
 Then
 \begin{equation}\label{7.2-0}
 |Du(x_0)|
 \le \frac{C}{\delta^2}
 \left(\fint_{Q(x_0, R)} | D u|^2\, dx \right)^{1/2},
 \end{equation}
 where $C$ depends only on $d$, $\kappa_1$, $\kappa_2$, $\omega$, and $(\sigma, M_0)$.
 \end{thm}
 
 \begin{proof}
Suppose $x_0\in F_k\subset Q(x_0, 2)$ for some $k$.
It follows from \eqref{local-2-0b} and interior estimates that
\begin{equation}\label{7.2-1}
|\nabla u(x_0)|
\le  C\delta^{-2} \left(\int_{Q(x_0, 2)} |\nabla u|^2 \, dx \right)^{1/2}
+ C \left(\int_{F_k} |\nabla u|^2 \, dx \right)^{1/2}.
\end{equation}
Choose $\phi\in \mathcal{R}$ such that $u-\phi \perp \mathcal{R}$ in $H^1(F_k; \R^d)$.
Since $u-\phi$ satisfies the same conditions as $u$, we may use \eqref{7.2-1} with $u-\phi$ in the place of $u$.
As a result, 
$$
\aligned
|Du(x_0)| & \le |\nabla (u-\phi) (x_0) |\\
& \le  C\delta^{-2} \left(\int_{Q(x_0, 2)} |\nabla u|^2 \, dx \right)^{1/2}
+ C \left(\int_{F_k} | D u|^2 \, dx \right)^{1/2},
\endaligned
$$
where we have used the second Korn inequality as well as the fact 
$|\nabla \phi|\le C \|\nabla u\|_{L^2(F_k)}$.
This, together with \ref{7.1-0} with $f=0$, gives
\begin{equation}\label{7.2-2}
\aligned
|D u(x_0)|
 & \le C \delta^{-2}
\left (\fint_{Q(x_0, 5)} |\nabla u|^2\right)^{1/2}\\
& \le   C \delta^{-2}
\left (\fint_{Q(x_0, R)} |\nabla u|^2\right)^{1/2},
\endaligned
\end{equation}
where we have used Theorem \ref{main-thm-1} for the last inequality.
Choose $\psi$  in  $\mathcal{R}$ so that
$$
\| \nabla  (u-\psi)\|_{L^2(Q(x_0, R)) } \le C \| D u \|_{L^2(Q(x_0, R))}.
$$
It  follows that
$$
\aligned
|Du (x_0)|
 & = |D (u-\psi) (x_0)|\\
 & \le C 
  \delta^{-2}
\left (\fint_{Q(x_0, R)} |\nabla (u-\psi) |^2\right)^{1/2}
\le C\delta^{-2}
 \left (\fint_{Q(x_0, R)} | D  u|^2\right)^{1/2}.
\endaligned
$$
where we have used \eqref{7.2-2} with $u-\psi$ in the place of $u$.
 \end{proof}
 
 \begin{cor}
 Let $1\le \delta<\infty$. Then 
 \begin{equation}\label{7.3-0}
\Lambda_{\delta^2} (x) 
|D_x \Gamma_\delta (x, y)|
+ \Lambda_{\delta^2} (y)
|D_y \Gamma_\delta (x, y)| 
\le C |x-y|^{1-d}
\end{equation}
for any $x, y\in \R^d$ with $|x-y|_\infty\ge 4$.
 \end{cor}
 
 \begin{proof}
Since $\delta\ge 1$, it follows by Theorems \ref{thm-f-1} and \ref{thm-f-1a} that
$$
|\nabla_x \Gamma_\delta  (x, y)| + |\nabla_y \Gamma_\delta   (x, y)| \le
C |x-y|^{1-d}.
$$
for any $x, y\in \R^d$ and $x\neq y$.
 This, together with Theorem \ref{thm7.2}, gives \eqref{7.3-0}.
 \end{proof}
 
 %%%%%%%%%%%%%%%%%%%%%%%%%%%%%%%%%%
 
 \section{Lipschitz estimates for $\mathcal{L}_\delta (u)=f$} \label{section-7}
 
 The goal of this section  is to prove \eqref{f-Lip-f}.
 The case $0<\delta<\infty$ follows readily from Theorem \ref{main-thm-2} and estimates of fundamental solutions in Section \ref{section-6}.
 To handle the cases $\delta=0$ and  $\delta =\infty$, we use an approximation  argument.
 
 Let $\Omega$ be a bounded Lipschitz domain in $\R^d$.
 We call $\Omega$ a type II domain (with respect to $\omega$)  if $F_k \cap \Omega\neq \emptyset$ implies that 
 $\widetilde{F}_k \subset \Omega$.
 In particular, if $\Omega$ is a type II Lipschitz domain, then $\partial \Omega \cap \partial \omega =\emptyset$.
 
 \begin{lemma}\label{lemma-app-1}
Assume that $A$ and $\omega$ satisfy the same conditions as in Theorem \ref{main-thm-1}.
 Let $0<\delta\le 1$ and $\Omega$ be a type II Lipschitz domain.
 Let $u_\delta \in H^1(\Omega; \R^d)$ be  a weak solution of
 $\mathcal{L}_\delta (u_\delta)=f \chi_\omega $ in $\Omega$ and $u_0\in H^1(\Omega; \R^d)$ a weak solution of
 $\mathcal{L}_0 (u_0)=f\chi_\omega $ in $\Omega$, where $f\in L^2(\Omega; \R^d)$.
 Suppose that $u_\delta=u_0$ on $\partial\Omega$.
 Then
 \begin{equation}\label{app-1-0}
 \| u_\delta -u_0\|_{H^1(\Omega)}
 \le C \delta \| D u_0\|_{L^2(\Omega\cap F)},
 \end{equation}
 where $C$ depends only on $d$, $\kappa_1$, $\kappa_2$, and $\omega$.
 \end{lemma}
 
 \begin{proof}
 Let $w=u_\delta -u_0\in H^1_0(\Omega; \R^d) $.
 Since $u_0$ is a weak solution of $\mathcal{L}_0 (u_0)=f\chi_\omega $ in $\Omega$,
 \begin{equation}\label{app-1-1}
 \int_{\Omega\cap \omega} A\nabla u_0 \cdot \nabla v \, dx
 =\int_{\Omega\cap \omega} f \cdot v\, dx
 \end{equation}
 for any $v\in H^1_0(\Omega; \R^d)$.  We also assume that $-\text{\rm div}(A\nabla u_0)=0$ in $F\cap \Omega$.
 Using
 $$
 \int_{\Omega\cap \omega} A\nabla u_\delta \cdot \nabla v\, dx
 + \delta^2 \int_{\Omega\cap F} A\nabla u_\delta \cdot \nabla v\, dx
 =\int_{\Omega\cap \omega} f\cdot v\, dx,
 $$
 we obtain 
 $$
 \int_{\Omega\cap \omega}
 A\nabla w \cdot \nabla w\, dx
 + \delta^2 \int_{\Omega\cap F} A \nabla w \cdot \nabla w 
 =-\delta^2 \int_{\Omega\cap F} A\nabla u_0 \cdot \nabla w\, dx.
 $$
 Hence, by \eqref{ellipticity} and  the Cauchy inequality, 
 \begin{equation}\label{app-1-2}
 \int_{\Omega\cap \omega}
 |D w|^2\, dx
 + \delta^2 \int_{\Omega\cap F} |D w|^2\, dx
 \le C \delta^2 \int_{\Omega\cap F} |D u_0|^2\, dx.
 \end{equation}
 Note that $\text{\rm div}(A\nabla w)=0$ in $F_k$ for any $F_k\subset \Omega$.
 By Lemma \ref{lemma-2.2} we have 
 $$
 \| D w\|_{L^2(F_k)}\le C \| D w \|_{L^2(\widetilde{F}_k \setminus \overline{F}_k)}.
 $$
 As a result, $\| Dw\|_{L^2(\Omega\cap F)}\le C \| Dw \|_{L^2(\Omega\cap \omega)}$,
 where we have used the assumption that $\Omega$ is a type II Lipschitz domain.
 This, together with \eqref{app-1-2} and the first Korn inequality \cite[p.13]{OSY-1992}, gives \eqref{app-1-0}.
 \end{proof}
 
 \begin{lemma}\label{lemma-app-2}
 Assume that $A$ and $\omega$ satisfy the same conditions as in Theorem \ref{main-thm-1}.
 Let $1\le  \delta<\infty$ and $\Omega$ be a type II Lipschitz domain.
 Let $u_\delta \in H^1(\Omega; \R^d)$ be  a weak solution of
 $\mathcal{L}_\delta (u_\delta)=f $ in $\Omega$ and $u_\infty\in H^1(\Omega; \R^d)$ a weak solution of
 $\mathcal{L}_\infty (u_\infty)=f$ in $\Omega$, where $f\in L^2(\Omega; \R^d)$.
 Suppose that $u_\delta=u_\infty$ on $\partial\Omega$.
 Then
 \begin{equation}\label{app-2-0}
 \| u_\delta -u_\infty \|_{H^1(\Omega)}
 \le C \delta^{-1} \big\{ 
  \| D u_\infty\|_{L^2(\Omega\cap \omega)}
  + \| f\|_{L^2(\Omega)} \big\},
 \end{equation}
 where $C$ depends only on $d$, $\kappa_1$, $\kappa_2$, and $\omega$.
 \end{lemma}

\begin{proof}
Since  $u_\infty$ is a weak solution of $\mathcal{L}_\infty (u_\infty) =f$, it follows that 
$$
\int_{\Omega\cap \omega} A\nabla u_\infty \cdot \nabla v\, dx
=\int_\Omega f \cdot  v\, dx
$$
for any $v\in H^1_0(\Omega; \R^d)$ with $Dv=0$ in $\Omega\cap F$, and  that 
$D u_\infty=0$ in $\Omega\cap F$.
Let $\phi \in H^1_0(\Omega; \R^d)$.
For each $F_k \subset \Omega$ and $g_k \in \mathcal{R}$, let $w_k \in H^1_0(\widetilde{F}_k; \R^d)$ be an extension 
of $\phi-g_k$ from $F_k$ to $\widetilde{F}_k$ with the property that
$$
\| w_k \|_{H^1(\widetilde{F}_k)} \le C \| \phi- g_k\|_{H^1(F_k)}.
$$
Extend $w_k$ from $\widetilde{F}_k$ to $\R^d$ by zero, and define 
$$
v=\phi -\sum_k w_k, 
$$
where the sum is taken over those $k$'s for which $F_k\subset \Omega$.
Note that $v\in H^1_0(\Omega; \R^d)$ and $v =g_k$ on $F_k$. Since $\Omega$ is a type II domain,
it follows that $Dv =0$ on $\Omega\cap F$.
As a result, we see that 
$$
\aligned
 & \Big|
\int_{\Omega\cap \omega}  A \nabla u_\infty\cdot \nabla \phi\, dx -\int_{\Omega} f \cdot \phi \, dx \Big|\\
&
= \Big| \int_{\Omega\cap \omega}  A \nabla u_\infty\cdot   \sum_k  \nabla w_k \, dx -\int_{\Omega} f \cdot \sum_k w_k \, dx \Big|\\
&\le C \sum_k \| D u_\infty\|_{L^2(\widetilde{F}_k \setminus \overline{F}_k)}
\| D w_k \|_{L^2(\widetilde{F}_k)}
+ C \sum_k \| f\|_{L^2(\widetilde{F}_k)}
\| w_k \|_{L^2(\widetilde{F}_k)}\\
&\le C \sum_k
\big( \| D u_\infty\|_{L^2(\widetilde{F}_k \setminus \overline{F}_k)}
+ \| f\|_{L^2(\widetilde{F}_k)}\big) \| \phi -g_k \|_{H^1(F_k)}\\
& \le C \big( \| D u_\infty\|_{L^2(\Omega\cap \omega)}
+\| f\|_{L^2(\Omega)} \big) \| D \phi\|_{L^2(\Omega\cap F)},
\endaligned
$$
where we have chosen $g_k \in \mathcal{R}$ so that
$\| \phi - g_k \|_{H^1(F_k)} \le  C \| D \phi\|_{L^2(F_k)}$.
This, together with 
$$
\int_{\Omega\cap \omega} A\nabla u_\delta \cdot \nabla \phi\, dx
+ \delta^2 \int_{\Omega\cap F} A \nabla u_\delta \cdot \nabla \phi\, dx
=\int_\Omega f \cdot \phi, dx,
$$
implies that
$$
\aligned
 & \Big| \int_{\Omega\cap \omega} A\nabla (u_\delta -u_\infty) \cdot \nabla \phi\, dx
+ \delta^2 \int_{\Omega\cap F} A \nabla (u_\delta -u_\infty)  \cdot \nabla \phi\, dx \Big|\\
&\le 
C \big( \| D u_\infty\|_{L^2(\Omega\cap \omega)}
+\| f\|_{L^2(\Omega)} \big) \| D \phi\|_{L^2(\Omega\cap F)}.
\endaligned
$$
By letting $w=u_\delta -u_\infty$ and $ \phi=w$ in the inequality above, we obtain
$$
\| D w \|_{L^2(\Omega\cap \omega)}
+\delta \| D w\|_{L^2(\Omega\cap F)}
\le C \delta^{-1} \big\{ \| D u_\infty\|_{L^2(\Omega\cap \omega)}
+ \| f\|_{L^2(\Omega)} \big\},
$$
where we have also used the Cauchy inequality.
Since $\delta\ge 1$ and $w\in H^1_0(\Omega; \R^d)$, the estimate \eqref{app-2-0}
follows by the first Korn inequality.
\end{proof}

\begin{thm}\label{main-thm-3}
Let $ d\ge 3$.
Assume that $A$ satisfies \eqref{ellipticity}, \eqref{smoothness}, and is 1-periodic.
Also assume that each $F_k$ is a bounded $C^{1, \sigma}$ domain.

\begin{enumerate}

\item

Let  $0\le \delta < 1$.
Let $u\in H^1(\Omega; \R^d)$ be a weak solution of
$\mathcal{L}_\delta (u)=f\chi_\omega$ in $Q(x_0, R)$ for some $R\ge 6$, where  $f\in L^p(Q(x_0, R);  \R^d)$ for some $p>d$.
Then
\begin{equation}\label{lip-f-8}
|\nabla u (x_0)|
\le   C \left(\fint_{Q(x_0, R)\cap \omega } |\nabla u|^2 \right)^{1/2}
+C R \left(\fint_{Q(x_0, R)} |f|^p \right)^{1/p},
\end{equation}
where $C$ depends only on $d$, $\kappa_1$, $\kappa_2$, $p$, $\omega$, and $(\sigma, M_0)$ in \eqref{smoothness}.

\item

Let $1\le  \delta \le \infty$.
Let $u\in H^1(\Omega; \R^d)$ be a weak solution of
$\mathcal{L}_\delta (u)= f$ in $Q(x_0, R)$ for some $R\ge 6$, where  $f\in L^p(Q(x_0, R);  \R^d)$ for some $p>d$.
Then
\begin{equation}\label{lip-f-9}
|\nabla u (x_0)|
\le  \left(\fint_{Q(x_0, R)\cap \omega } |\nabla u|^2 \right)^{1/2}
+C R \left(\fint_{Q(x_0, R)} |f|^p \right)^{1/p},
\end{equation}
where $C$ depends only on $d$, $\kappa_1$, $\kappa_2$, $p$, $\omega$, and $(\sigma, M_0)$.

\end{enumerate}
\end{thm}

 \begin{proof}
 
 By translation we may assume $x_0=0$.
 We consider 4 cases.
 
 \medskip
 
 \noindent{\bf Case 1.} Assume $1\le \delta  < \infty$.
 If $f=0$, this  is given by Theorem \ref{main-thm-2}.
 In general, let
 \begin{equation}\label{lip-f-10}
 v(x) =\int_{Q_R} \Gamma_\delta (x, y) f(y)\, dy.
 \end{equation}
 Then $\mathcal{L}_\delta (v)=f$ in $Q_R$, and by Theorems \ref{thm-f-1} and \ref{thm-f-1a},
 \begin{equation}\label{lip-f-11}
 \|\nabla v\|_{L^\infty (Q_R)} \le CR  \left(\fint_{Q_R} |f|^p\right)^{1/p}
 \end{equation}
for $p>d$.  Hence,
 $$
 \aligned
 |\nabla u(0)|
 &\le |\nabla (u-v)(0)| +|\nabla v(0)|\\
 & \le C \left(\fint_{Q_R\cap \omega } |\nabla (u-v)|^2\right)^{1/2}
 + C R  \left(\fint_{Q_R} |f|^p\right)^{1/p}\\
 &\le C \left(\fint_{Q_R\cap \omega } |\nabla u|^2 \right)^{1/2}
+C R \left(\fint_{Q_R  } |f|^p \right)^{1/p},
\endaligned
$$
where we have used the fact $\mathcal{L}_\delta (u-v)=0$ in $Q_R$.
 
 \medskip
 
 \noindent{\bf Case 2.} 
 Assume $\delta=\infty$.
 In this case we use an approximation argument.
 Choose a type II Lipschitz domain $\Omega$ such that $Q_{R-2}\subset \Omega\subset Q_R$.
 Let $u_\delta\in H^1(\Omega: \R^d)$ be a weak solution of $\mathcal{L}_\delta (u_\delta)=f$ in $\Omega$
 such that $u_\delta=u$ on $\partial\Omega$.
 It follows by Lemma \ref{lemma-app-2} that $u_\delta \to u$ in $H^1(\Omega; \R^d)$,  as $\delta \to \infty$.
 By the proof for Case 1, 
 $$
 \left(\fint_{Q_r} |\nabla u_\delta|^2\right)^{1/2}
 \le C \left(\fint_{Q_{R-2}\cap \omega } |\nabla u_\delta |^2 \right)^{1/2}
+C R \left(\fint_{Q_R} |f|^p \right)^{1/p},
 $$
 for  $r\in (0, 1/4)$.
 The proof is complete by letting $\delta\to \infty$ and then $r\to 0$ in the inequality above.
 
 \medskip
 
 \noindent{\bf Case 3.}
 Assume  $0<\delta< 1$.
 If $f=0$, the estimate \eqref{lip-f-8} is given by   Theorem \ref{main-thm-2}.
 In general, let
 $$
 v(x)=\int_{Q_R\cap \omega} \Gamma_\delta (x, y) f(y)\, dy.
 $$
 Then
 $
 \mathcal{L}_\delta (v)=f\chi_\omega$ in $Q_R$, and by Theorems \ref{thm-f-1} and \ref{thm-f-3},
 $$
 \|\Lambda_\delta  \nabla  v\|_{L^\infty(Q_R)}
 \le C R \left(\fint_{Q_R} |f|^p \right)^{1/p}.
 $$
 Observe that since $\text{\rm div}(A\nabla v)=0$ in $Q_R\cap F$, it follows from Theorem \ref{local-thm-0}  that
 $$
 \aligned
 |\nabla v (0)|
  & \le C \left( \fint_{Q_1} |\nabla v|^2 \right)^{1/2}
 + C \left(\fint_{Q_1} |f|^p \right)^{1/p}\\
&  \le C  \left(\fint_{Q_4\cap \omega} |\nabla v|^2 \right)^{1/2}
 + C \left(\fint_{Q_1} |f|^p \right)^{1/p},
\endaligned
$$
 where we have used \eqref{2.2-0} for the last inequality.
 Hence,
 $$
 \aligned
  |\nabla u (0)|
 &\le |\nabla (u-v)(0)| +  |\nabla v(0)|\\
 &\le C \left(\fint_{Q_R\cap \omega} |\nabla (u-v)|^2\right)^{1/2}
 + C R \left(\fint_{Q_R} |f|^p \right)^{1/p}\\
 & \le C \left(\fint_{Q_R\cap \omega} |\nabla u|^2\right)^{1/2}
 +C R \left(\fint_{Q_R} |f|^p \right)^{1/p}.
 \endaligned
 $$
 
 \noindent{\bf Case 4.}
 Assume $\delta=0$.
 As in Case 2,  this follows from Case 3 by using  the  approximation in Lemma \ref{lemma-app-2}.
 We omit the details.
\end{proof} 
 
 \bibliographystyle{amsplain}
 
\bibliography{S20202.bbl}

\bigskip

\begin{flushleft}

Zhongwei Shen,
Department of Mathematics,
University of Kentucky,
Lexington, Kentucky 40506,
USA.

E-mail: zshen2@uky.edu
\end{flushleft}

\bigskip

\end{document}